\documentclass[a4paper,abstract=true,DIV=default,toc=index]{scrartcl}

\usepackage[normalem]{ulem}

\usepackage{amsmath}
\usepackage{amsthm}
\usepackage{amsfonts}
\usepackage[noadjust]{cite}
\usepackage{color,enumitem,graphicx}

\usepackage{mathtools} 

\usepackage[english]{babel}
\usepackage[dvipsnames]{xcolor}

\usepackage{authblk} 

\usepackage{url}
\usepackage[colorlinks=true,urlcolor=blue,citecolor=red,linkcolor=blue,
linktocpage,pdfpagelabels,
bookmarksnumbered,bookmarksopen,pagebackref]{hyperref}

\usepackage[T1]{fontenc}
\usepackage{lmodern}

\usepackage[useregional]{datetime2}

\usepackage{fixmath}

\usepackage{todonotes}

\newtheorem{Th}{Theorem}[section]

\newtheorem{Lem}[Th]{Lemma}

\theoremstyle{definition}

\newtheorem{Rem}[Th]{Remark}
\newtheorem{Ex}[Th]{Example}

\newcommand{\wt}{\widetilde}

\newcommand{\eps}{\varepsilon}

\newcommand{\R}{\mathbb{R}}

\newcommand{\Z}{\mathbb{Z}}

\newcommand{\cE}{{\mathcal E}}

\newcommand{\cS}{{\mathcal S}}

\newcommand{\weakto}{\rightharpoonup}

\numberwithin{equation}{section}

\newcommand{\rad}{\mathrm{rad}}

\allowdisplaybreaks

\begin{document}

\nocite{*}

\title{Compact embeddings for weighted fractional Sobolev spaces and applications to Nonlinear Schrödinger Equations}

\author[1]{Federico Bernini\thanks{Email address: \texttt{f.bernini@univpm.it}}}

\author[2]{Sergio Rolando\thanks{Email address: \texttt{sergio.rolando@unito.it}}}

\author[2]{Simone Secchi\thanks{Email address: \texttt{simone.secchi@unimib.it}}}

\affil[1]{\small Dipartimento di Ingegneria Industriale e Scienze Matematiche, Università Politecnica delle Marche,	Via Brecce Bianche, 12, 60131 Ancona, Italy}    

\affil[2]{\small Dipartimento di Matematica e Applicazioni, Università degli Studi di Milano-Bicocca, Via Roberto Cozzi 55, I-20125, Milano, Italy}

\renewcommand\Authands{ and }

\date{}
\maketitle

\begin{abstract} 
The aim of this work is to prove a compact embedding for a weighted fractional Sobolev spaces. As an application, we use this embedding to prove, via variational methods, the existence of solutions for the following Schr\"odinger equation
\[
	(-\Delta)^su + V(|x|)u = K(|x|)f(u), \quad \text{ in } \R^N,
\]
where the two measurable functions $K > 0$ and $V \geq 0$ could vanish at infinity.
\medskip

\noindent \textbf{Keywords:} compact embeddings, vanishing potential, fractional Schr\"odinger equation.
   
\noindent \textbf{AMS Subject Classification:}  46E35, 
												35R11, 
												35J20. 
\end{abstract}

\tableofcontents

\section{Introduction}
When we look for critical points of a given (smooth) functional, the typical goal is to prove a compactness property for a suitable sequence, usually a Palais-Smale sequence (\cite{PaSm1964}) (or one of its most famous generalizations, the Cerami sequence \cite{Cerami1978}), for a given functional, which is defined in an ambient space endowed with a topology that is poor enough to allow the existence of convergent subsequence, namely proving the validity of the Palais-Smale condition (or Cerami condition), but rich enough to gurantee a good regularity for the functional.

In bounded domains, a classical strategy to recover compactness for Palais-Smale sequences is to apply the Rellich-Kondrachov embedding Theorem for Sobolev spaces. However, compactness property is as important as it is easy to not have it, and reasons for which one has to face a lack of compactness are many: this could happen when dealing with equations where the nonlinearity has Sobolev critical growth (we refer to the pioneering work of Brézis-Nirenberg \cite{BrNi} and the references therein), or with problems considered in the whole space $\R^N$. In this last case, we talk about \textit{entire problems} and the loss of compactness is given by the invariance of the equation under the group of translations. Moreover, in both cases above, Rellich-Kondrachov Theorem can not be used.

To overcome this problem, in 1977 Strauss (\cite{Strauss1977}) was able to generalize the embedding Theorem to the whole space $\R^N$ with $N \geq 2$, as long as one restricts to the subspace of radial functions: the proof relies on a decaying estimates for radial functions. This result has been further generalized in 1983 by Berestycki and Lions (see \cite{BeLi1983-1,BeLi1983-2}).

Therefore, in the entire space the variational scheme can be recovered for a subspace (namely the subspace of radial functions) of the ambient space. However, it turns out that this is not too restrictive, because of another important result due to Palais (see \cite{Palais1979}): the Principle of symmetric criticality. Thanks to this principle, under suitable assumptions it is possibile to show that a critical point for a functional restricted to the subspace of the radial functions is actually a free critical point: in other words, the radiality constraint is a sort of a natural constraint.


Summarizing so far, by Strauss' embedding Theorem and Palais' Principle of symmetric criticality, the whole variational scheme used in bounded domain can be recovered also for entire problems. This fact allowed the study of many equations otherwise difficult to handle, using some of the by-now classic results in Critical Point Theory, such as the Mountain-Pass Theorem by Ambrosetti-Rabinowitz (\cite{AmRa1973}), Linking Theorem by Rabinowitz (\cite{Rabinowitz1978}) or the Nehari manifold method (\cite{Nehari1960}) and in particular, great interest was given to the study of Schr\"odinger type equations of the form
\[
	-\Delta u + V(|x|)u = K(|x|)f(|x|,u), \text{ in } \R^N.
\]
The literature for this equation is very huge and impossible to list completely (and this is not the goal of this paper). However, for our purpose we cite some important results in this direction, due to Benci-Fortunato (see \cite{BeFo2004-1,BeFo2004-2}), where they assume the so called \textit{double-power type growth condition} for the nonlinearity. This assumption has then been considered in some later works by many authors (\cite{Azzollini2008,AzPiPo2011,AzPo2008,BaBeRo2007,BaBeRo2009,BeFo2004-1,BeFo2004-2,BaGuRo2021,BaPiRo2011,BeLi1983-1,BeLi1983-2}) in the study of semilinear Maxwell equations. To better deal with this kind of growth, it was observed that it is convenient to embed the Sobolev ambient space into the sum of weighted Lebesgue spaces $L^p_K$, where the weight $K$ is the one in front of the nonlinearity. A very deep treatment on the sum of weighted Lebesgue spaces has been done in \cite{BaPiRo2011}. Later on, this theory has been generalized by \cite{BaGuRo2015-1}, were the authors obtained a compact embedding results for the radial subsapace into the sum of weighted Lebesgue spaces. Using this embedding, they then showed (see \cite{BaGuRo2015-2}) the validity of a Principle of Symmetric Criticality, proving the existence and the multiplicity of solutions for the following equation
\begin{equation*}
	-\Delta u + V(|x|)u = K(|x|)f(u), \text{ in } \R^N,
\end{equation*}
where $N \geq 3$, $V \geq 0$, $K>0$ and $f \in C(\R)$ is such that $f(0)=0$. We want to point out that in those works, the authors allowed the potential $V$ to vanish at infinity, in the sense that
\begin{equation}
\label{vanishing:potential}
	\lim_{|x| \to +\infty} V(|x|)=0.
\end{equation}

This kind of request is of recent interest: in fact, most of the works made in the previous decades, required potentials that did not vanish to infinity: in fact, if \eqref{vanishing:potential} holds, it can be seen that the natural weighted Sobolev space $H_V^1$ is not contained in $L^q_K$ anymore. That is, the variational scheme become useless. For this reason, the sum of (weighted) Lebesgue spaces seems to be the right way to deal with this kind of problems.

The result proved in \cite{BaGuRo2015-1} by Badiale \textit{et al.} generalizes some previous works by Su and Tian (\cite{SuTi2012}) and Su, Wang and Willem (\cite{SuWaWi2007-1,SuWaWi2007-2}) and it is complementary to the work of Bonheure and Mercuri (\cite{BoMe2011}), where the authors proved a Sobolev embedding into a single weighted Lebesgue space by means of the so-called \textit{Hardy-Dieudonné comparison class}.

As already pointed out above, compact embeddings are really useful to find solutions for a wide class of PDEs: in particular, nonlinear Schr\"odinger equations has been widely studied in the last years both in the local case (\cite{AmFeMa2005,AmMaSe2001,BaWa1995,dPFe1996,Rabinowitz1992}) that in the nonlocal setting (\cite{Ambrosio2016,DPPaVa2013,FeQuTa2012,Secchi2013,Secchi2016}). 
For other papers dealing with vanishing potential see  \cite{AlAlMe2014,AlSo2013,AmRu2006,AmWa2005,CaCePeUb2021}. More recently, a nonlinear Schr\"odinger equation, where the potential $V$ satisfies \eqref{vanishing:potential} was considered in \cite{ApMBSe} on a Cartan-Hadamard manifold.
Later on, other nonlinearities with more general assumptions (like Sobolev subcritical growth) have also been considered. In here, we are going to consider nonlinearity satisfying the \textit{double-power type growth condition} (see below for the formal definition). 

Inspired by \cite{BaGuRo2015-1, BaGuRo2015-2}, in this work we are going to extend their results to the nonlocal case: in particular, we are dealing with the following fractional equation
\begin{equation}
	\label{1}
	(-\Delta)^su + V(|x|)u = K(|x|)f(u), \quad \text{ in } \R^N
\end{equation}
where $N \geq 2$, $V \geq 0$ and $K>0$ are given potentials, and $f$ satisfies suitable hypoteses (see below).

As in the local case, also in the nonlocal one the compactness embedding proof relies on a Strauss type estimate, but, due to the nature of the problem, for weighted Sobolev spaces. This estimate was proved by De N\'{a}poli in \cite{DeNapoli2018}, generalizing the result obtained by Snitzoff in the local case (\cite{Snitzoff2003}). De N\'{a}poli has then used his result to study the symmetry breaking phenomenon for an equation driven by fractional Laplace operator.

Our main result states that, under assumptions to be specified later,  the embedding
\[
	H^s_{V,\rad}(\mathbb{R}^N)  \hookrightarrow L^{q_1}_K (\mathbb{R}^N) + L^{q_2}_K (\mathbb{R}^N)
\]
is compact for suitable $q_1>1$, $q_2 > 1$ (see Theorems \ref{th:comp:emb}, \ref{S0:holdness} and \ref{Sinf:holdness} below).

As remarked by De N\'{a}poli, since the Strauss' type inequality implies the existence of a continuous (outside the origin) representative, his inequality holds for $s>1/2$ only. This could seem to be a strong restriction in the result, but actuallyit  is not. In fact, this restriction is somehow \textit{structural}, in the sense that functions in $H^s (\mathbb{R}^N)$, with $s$ ``small'', do not have  pointwise representatives. This is confirmed by the Sobolev embedding theorem, which states that $H^s(\mathbb{R}^N)$ is continuously embedded into $C^{0,s-N/2}(\R^N)$ (see \cite{Ambrosio2020,DiPaVa2012}). In the nonlocal case decaying estimates are not expected if $s \to 0^+$, because the space $H^2(\mathbb{R}^N)$ gets closer and closer to $L^2(\mathbb{R}^N)$. Moreover, although some results for $s<1/2$ could be interesting, getting closer to $0$ means we are going farest from a  space of differentiable functions, and this makes the analysis more delicate and subtle.

Going back to results of this paper, exploiting the nonlocal Sobolev compact embedding, we are able to provide existence of solutions for equation \eqref{1}: in particular, the solutions turn out to be of Mountain-Pass type. Additionally, if the nonlinearity enjoys a symmetry assumption (namely if $f$ is odd) then we can prove that equation \eqref{1} has infinitely many solutions: this can be done using the $\Z^2$-version of the Mountain-Pass (\cite{AmRa1973},\cite{Rabinowitz1974}).

\bigskip

The structure of this paper is as follows. In Section \ref{Preliminaries:Sec} we recall the definitions of the fractional Laplacian and the Sobolev space $H^s(\mathbb{R}^N)$. Then we introduce the weighted spaces we will mainly work in, also stating some related properties and inequalities we will need in the paper. Section \ref{Compacntess:Sec} is devoted to the proof of our compactness result (Theorem \ref{th:comp:emb}). Since the assumptions of this result are not straightforward to check in the applications, in Section \ref{Operative:way:Sec} we provide two theorems (Theorem \ref{S0:holdness} and Theorem \ref{Sinf:holdness}) which allow to obtain the compactness via easier hypotheses. In Section \ref{Existence:Sec} we apply our previous results, stating and proving an existence result for the weighted fractional Schr\"odinger equation \eqref{1}. Finally, in Section \ref{Examples:Sec} we provide some examples of concrete situations allowed by our result.

%

\section{Preliminaries and variational setting}
\label{Preliminaries:Sec}
In this Section we recall some known facts about fractional Laplacian and we introduce the spaces we are going to deal with. We start with the operator, focusing only on the facts we need and referring the reader to \cite{DiPaVa2012},\cite{MoRaSe2016} and \cite{Ambrosio2020} for a complete treatment.

\bigskip

Let $0 < s < 1$ and let $u$ be a rapidly decreasing function, i.e., 
\begin{displaymath}
	u \in \cS=\left\{u \in C^{\infty}(\R^N) : \sup_{x \in \R^N}|x^{\alpha}D^{\beta}u(x)| < +\infty,\ \hbox{for all multi-indices $\alpha$, $\beta$} \right\}.
\end{displaymath}
 The fractional Laplace operator is defined as
\begin{equation}
\label{frac:lap}
\begin{aligned}
	(-\Delta)^su(x) &= \frac{C(N,s)}{2}\lim_{\eps \to 0^+}\int_{B^c_{\eps}(x)} \frac{u(x)-u(y)}{|x-y|^{N+2s}} \, dy \\
	&= \frac{C(N,s)}{2}\operatorname{P.V.} \int_{\R^N} \frac{u(x)-u(y)}{|x-y|^{N+2s}} \, dy,
\end{aligned}
\end{equation}
where $\operatorname{P.V.}$ denotes the Cauchy principal value and
\[
	C(N,s) = \left(\int_{\R^N} \dfrac{1-\cos(\zeta_1)}{|\zeta|^{N+2s}} \, d\zeta\right)^{-1} = 2^{-(N+2s)/2+1}\pi^{-\frac{N}{2}}2^{2s}\frac{s(1-s)}{\Gamma(2-s)}.
\]
is a normalization constant (see \cite[Proposition 3.3]{DiPaVa2012}).

The operator defined in \eqref{frac:lap} is stricly related to the fractional Sobolev space defined as
\[
	H^s(\R^N)=\left\{u \in L^2(\R^N) : [u]_{H^s(\R^N)} < +\infty \right\},
\]
where
\[
	[u]_{H^s(\R^N)}=\left(\frac{C(N,s)}{2}\int_{\R^N}\int_{\R^N} \frac{|u(x)-u(y)|^2}{|x-y|^{N+2s}} \, dx \, dy \right)^{\frac{1}{2}}
\]
is the Gagliardo seminorm. This space is a  Hilbert space with respect to the inner product
\[
	(u,v)_{H^s(\R^N)} = \int_{\R^N} u(x)v(x) \, dx + \frac{C(N,s)}{2}\int_{\R^N \times \R^N} \frac{(u(x)-u(y))(v(x)-v(y))}{|x-y|^{N+2s}} \, dx \, dy
\]
for every $u,v \in H^s(\R^N)$. The above inner product is associated to the  norm 
\begin{displaymath}
	\|u\|^2_{H^s(\R^N)} = \|u\|^2_{L^2(\R^N)} + [u]^2_{H^s(\R^N)},
\end{displaymath}

\begin{Rem}
	To ease notation, we will also write $H^s$, $L^p$, etc instead of $H^s(\mathbb{R}^N)$, $L^p(\mathbb{R}^N)$, etc. 
\end{Rem}
 An equivalent definition of $H^s(\R^N)$ can be given via Fourier transform (see \cite{DiPaVa2012}, Section 3.1).


Now, we are going to introduce the weighted spaces. For a given measurable function $V\geq 0$, we define the fractional weighted Sobolev space 
\[
	H^s_V(\R^N) = \left\{u \in \dot{H}^s(\R^N) : \int_{\R^N} V(|x|)|u|^2 \, dx < \infty\right\}
\]
endowed with the norm defined by
\[
	\|u\|^2:= \|u\|^2_{H^s_V(\R^N)} = [u]^2_{H^s(\R^N)} + \int_{\R^N} V(|x|)|u|^2 \, dx,
\]
where 
\[
\dot{H}^s(\R^N):=\left\{u \in L^{2^*_s}(\R^N) : [u]_{H^s} < \infty \right\},\qquad 2^*_s:=\frac{2N}{N-2s}. 
\]
It is well known that $C^{\infty}_c(\R^N)$ is dense in $\dot{H}^s(\R^N)$. For our purpose, we need the subspace of the radial functions of this space, that is
\[
	H^s_{V,\rad}(\R^N) = \left\{u \in H^s_V(\R^N) \mid u(x)=u(\vert x \vert) \right\}.
\] 

Moreover, given a measurable function $K>0$, we introduce the weighted Lebesgue space
\[
	L^q_K (\mathbb{R}^N) :=L^q\left(\R^N, K(|x|)\,dx\right)
\]
equipped with the norm
\[
	\|u\|_{L^q_K}=\left(\int_{\R^N}K(|x|)|u(x)|^q \, dx \right)^{\frac{1}{q}},
\]
and we recall that
\[
	L^{q_1}_K (\mathbb{R}^N) + L^{q_2}_K (\mathbb{R}^N) :=\left\{u=u_1+u_2 : u \in L^{q_1}_K(\R^N), u_2 \in L^{q_2}_K(\R^N) \right\}
\]
is a Banach space with respect to the norm
\[
	\|u\|_{L^{q_1}_K + L^{q_2}_K}:= \inf_{u=u_1+u_2}\max\left\{\|u_1\|_{L^{q_1}_K},\|u_2\|_{L^{q_2}_K} \right\}.
\]
For a precise description on the sum of Lebesgue spaces see \cite{BaPiRo2011}.

Concerning the potentials $V$ and $K$, we will consider the following assumptions:
\begin{itemize}
\item[(V)] $V \colon (0,+\infty) \to [0,+\infty)$ is a continuous function;
\item[(K)] $K \colon (0,+\infty) \to (0,+\infty)$ is a continuous function.
\end{itemize}


We conclude this Section by recalling two estimates that will be very helpful in the upcoming computations. 
The first one provides a control on the fractional critical norm. This result is contained in \cite[Theorem 6.5]{DiPaVa2012}, and it holds for measurable and compactly supported functions. Exploiting the density of $C^{\infty}_c(\R^N)$ in $\dot{H}^s(\R^N)$, we will state it in the case of functions belonging to $\dot{H}^s(\R^N)$.
\begin{Th}
	\label{crit_est}
	Let $0 < s < 1$. There exists constant $S:=S(N,s)>0$ such that
	\[
	\|u\|_{L^{2_s^*}} \leq S[u]_{H^s}, \text{ for all } u \in \dot{H}^s(\R^N).
	\]
\end{Th}

\begin{Rem} \normalfont
	The Sobolev constant $S$ is explicit and optimal and is given by the formula (cf. \cite[Theorem 1.1]{CoTa2004} and \cite{PaPi2014}):
	\begin{equation}
		\label{Sobolev}
			S:=S(N,s)=\left(\frac{1}{2^{2s}\pi^s} \frac{\Gamma(\frac{N-2s}{2})}{\Gamma(\frac{N+2s}{2})}\left[\frac{\Gamma(N)}{\Gamma(\frac{N}{2})} \right]^{2s/N} \right)^{2^*_s/2}.
	\end{equation}
\end{Rem}

\begin{Rem}\label{Lp_loc} \normalfont
	As a consequence of Theorem \ref{crit_est}, the space $H^s_V(\R^N)$ is continuously embedded both into $L^{2_s^*}(\R^N)$ and into $L^p_{loc}(\R^N)$ for every $p\in[1,2_s^*]$.
\end{Rem}

The next result is due to P.L. De N\'{a}poli \cite{DeNapoli2018} and provides a Strauss-type inequality in the fractional envirorment. Here we state it in the particular case where $a=0$ and $q=2^*_s$ (so that the exponent $\sigma$ of \cite[Theorem 3.1]{DeNapoli2018} becomes $\sigma=(N-2s)/2$).
\begin{Th}
\label{Strauss_Th}
Let $s > 1/2$ and set $\theta=(N-2s)/(2sN-2s)$. 
Then there exists a constant $C(N,s)$ such that for every $u \in H^s_{0,\rad}(\R^N)$ we have
\[
	|u(x)| \leq \frac{C(N,s)}{\vert x \vert^\frac{N-2s}{2}} [u]^{\theta}_{H^s}\|u\|^{1-\theta}_{L^{2^*_s}}
	\quad\textrm{for almost every }x\in\R^N.
\]
In particular, for all $u \in H^s_{V,\rad}(\R^N)$ we get
\begin{equation}
\label{Strauss_ineq}
	|u(x)| \leq \frac{C(N,s)}{\vert x \vert^\frac{N-2s}{2}} \|u\|_{H^s_V}
	\quad\textrm{for almost every }x\in\R^N.
\end{equation}
\end{Th}

\begin{Rem}\label{cpt_loc} \normalfont
	By a well know compactness lemma due to Strauss \cite{Strauss1977}, it easy to exploit the embeddings of Remark \ref{Lp_loc} and the pointwise estimate \eqref{Strauss_ineq} to deduce that the space $H^s_{V,\rad}(\R^N)$ is compactly embedded into $L^p_{loc}(\R^N)$ for every $p\in(1,2_s^*)$.
\end{Rem}

\section{Compactness}
\label{Compacntess:Sec}

We divide this Section in two parts: in the first one we provide the continuous embedding of $H^s_{V,\rad}(\R^N)$ into $L^{q_1}(\R^N) + L^{q_2}(\R^N)$, for some $q_1,q_2 > 1$. In the second part, we will show that this embedding is actually compact. The proofs rely on a careful analysis of the variational problems
\begin{equation}
\label{S0}
	\cS_0(q,R) = \sup\limits_{\substack{u \in H^s_{V,\rad} \\ \|u\|=1}} \int_{B_R} K(|x|)|u|^q \, dx
\end{equation}
\begin{equation}
\label{Sinf}
	\cS_{\infty}(q,R) = \sup\limits_{\substack{u \in H^s_{V,\rad} \\ \|u\|=1}} \int_{B^c_R} K(|x|)|u|^q \, dx.
\end{equation}
Observe that $\cS_0(q,R)$ is nondecreasing and $\cS_{\infty}(q,R)$ is nonincreasing in $R$.

We are going to show that if both \eqref{S0} and \eqref{Sinf} are finite, then we get the continuous embedding of $H^s_{V,\rad}(\R^N)$ into $L^{q_1}_K(\R^N) + L^{q_2}_K(\R^N)$, while we  the same embedding is compact provided that they are both vanishing (see Theorem \ref{th:cont:emb} and \ref{th:comp:emb} for the precise statements).

We begin by proving the following lemma.
\begin{Lem}
\label{annulus:comp:emb}
Let $s>1/2$, $R > r > 0$ and $q \in (1,+\infty)$. Then there exist a constant $\tilde C= \tilde C(N,R,r,q,s) > 0$ and two numbers $t>1$ and $1<\tilde q<q$ such that
\begin{equation}
\label{ann:estimate}
	\frac{\int_{B_R \setminus B_r} K(|x|)|u|^q \, dx}{\tilde C\|K(|\cdot|)\|_{L^t(B_R \setminus B_r)}} \leq \left(\int_{B_R \setminus B_r} |u|^2 \, dx \right)^{\frac{\tilde{q}-1}{2}}\|u\|^{1+q-\tilde{q}}\
\end{equation}
for all $u \in H^s_{V,\rad}(\R^N)$.
\end{Lem}
\begin{proof}
Define 
\[
\tilde q = 2\left( 1+ \frac{s}{N} - \frac{1}{t} \right) 
\]
and fix $t>(2^*_s)'$ such that $1<\tilde q<q$. This choice is possible under our assumptions. Indeed, $\tilde{q}>1$ is equivalent to $t>(2_s^*)'$. The inequality $\tilde{q}<q$ is equivalent to $(2N+2s-Nq)t < 2N$. Supposing that $2N+2s-Nq>0$ (otherwise this condition is obviously satisfied), we get the condition
\begin{displaymath}
	t< \frac{2N}{2N+2s-Nq}.
\end{displaymath}
However,
\begin{displaymath}
	\frac{2N}{2N+2s-Nq} > \left(2_s^* \right)'
\end{displaymath}
as soon as $q>1$.

By Theorem~\ref{crit_est} andthe  H\"older inequality first with exponents $2^*_s$ and $(2^*_s)'$, and then with exponents $l:=\frac{t}{(2^*_s)'}$ and $l'=\frac{t(N+2s)}{t(N+2s)-2N}$, we get
\begin{align*}
	&\int_{B_R \setminus B_r} K(|x|)|u|^{q-1}|u| \, dx \\
	& \leq \left(\int_{B_R \setminus B_r}\left(K(|x|)|u|^{q-1}\right)^{(2^*_s)'} \, dx \right)^{\frac{1}{(2^*_s)'}} \left(\int_{B_R \setminus B_r}|u|^{2^*_s} \, dx \right)^{\frac{1}{2^*_s}}\\
	& \leq \left(\int_{B_R \setminus B_r}\left(K(|x|)|u|^{q-1}\right)^{(2^*_s)'} \, dx \right)^{\frac{1}{(2^*_s)'}} S[u]_{H^s} \\
	& \leq \left[\left(\int_{B_R \setminus B_r}(K(|x|)^{(2^*_s)'l} \, dx \right)^{\frac{1}{l}}\left(\int_{B_R \setminus B_r}|u|^{(q-1)(2^*_s)'l'} \, dx \right)^{\frac{1}{l'}}\right]^{\frac{1}{(2^*_s)'}}S\|u\|.
\end{align*}
Since $(2^*_s)'l'=\frac{2Nt}{Nt+2st-2N}=\frac{2}{\tilde{q}-1}$, we then obtain
\begin{align*}
	&\int_{B_R \setminus B_r} K(|x|)|u|^{q} \, dx \\
	&\leq S \left[\left(\int_{B_R \setminus B_r}(K(|x|)^t \, dx \right)^{\frac{(2^*_s)'}{t}}\left(\int_{B_R \setminus B_r}|u|^{(q-1)\frac{2}{\tilde{q}-1}} \, dx \right)^{\frac{\tilde{q}-1}{2}(2^*_s)'}\right]^{\frac{1}{(2^*_s)'}}\|u\| \\
	& \leq S \|K(|\cdot|)\|_{L^t(B_R \setminus B_r)}  \left(\int_{B_R \setminus B_r} |u|^{2\frac{q-1}{\tilde{q}-1}} \, dx \right)^{\frac{\tilde{q}-1}{2}} \|u\| .
\end{align*}
%
%
Since $\tilde{q} < q$, using \eqref{Strauss_ineq} we now have
\begin{align*}
	&\int_{B_R \setminus B_r} K(|x|)|u|^q \, dx \leq S \|K(|\cdot|)\|_{L^t(B_R \setminus B_r)}  \left(\int_{B_R \setminus B_r} |u|^{2\frac{q-1}{\tilde{q} -1}-2}|u|^2 \, dx \right)^{\frac{\tilde{q} -1}{2}} \|u\| \\
	& \leq S \|K(|\cdot|)\|_{L^t(B_R \setminus B_r)}  \left(\int_{B_R \setminus B_r} \left(C(N,s)\|u\|r^{-\frac{N-2}{2}} \right)^{2\frac{q-1}{\tilde{q} -1}-2}|u|^2 \, dx \right)^{\frac{\tilde{q} -1}{2}} \|u\| \\
	& = S \|K(|\cdot|)\|_{L^t(B_R \setminus B_r)}  \left(C(N,s)\|u\|r^{-\frac{N-2}{2}} \right)^{q-\tilde{q} } \left(\int_{B_R \setminus B_r} |u|^2 \, dx \right)^{\frac{\tilde{q} -1}{2}} \|u\| \\
	& = \tilde C(N,R,r,q,s,t)\|K(|\cdot|)\|_{L^t(B_R \setminus B_r)} \left(\int_{B_R \setminus B_r} |u|^2 \, dx \right)^{\frac{\tilde{q} -1}{2}} \|u\|\|u\|^{q-\tilde{q} }
\end{align*}
and the proof is  complete.
\end{proof}

\subsection{Continuous embedding}
We are ready to prove a first embedding result.
\begin{Th}
\label{th:cont:emb}
Let $N \geq 2$, $s>1/2$, $q_1>1$, $q_2 > 1$ and $V$ and $K$ as in (V) and (K). If
\begin{equation}\tag{$\cS'$}
\label{S':hyp}
	\cS_0(q_1,R_1) < +\infty \text{ and } \cS_{\infty}(q_2,R_2) < +\infty \text{ for some } R_1,R_2 > 0,
\end{equation}
then $H^s_{V,\rad}(\R^N)$ is continuously embedded into $L^{q_1}_K(\R^N) + L^{q_2}_K(\R^N)$.
\end{Th}
\begin{proof}
Without loss of generality, we can assume $R_1 < R_2$. Let $u \in H^s_{V,\rad}(\R^N)$, $u \neq 0$. We have
\begin{equation}
\label{BR1:est}
	\int_{B_{R_1}} K(|x|)|u|^{q_1} \, dx = \|u\|^{q_1}\int_{B_{R_1}} K(|x|)\frac{|u|^{q_1}}{\|u\|^{q_1}} \, dx \leq \|u\|^{q_1}\cS_0(q_1,R_1)
\end{equation}
and
\begin{equation}
\label{BR2:comp:est}
	\int_{B^c_{R_2}} K(|x|)|u|^{q_2} \, dx = \|u\|^{q_2}\int_{B^c_{R_2}} K(|x|)\frac{|u|^{q_2}}{\|u\|^{q_2}} \, dx \leq \|u\|^{q_2}\cS_{\infty}(q_2,R_2)
\end{equation}
and both are finite thanks to assumption \eqref{S':hyp}.

We now estimate the integral on the annulus $B_{R_2} \setminus B_{R_1}$, by Lemma \ref{annulus:comp:emb} with $q=q_1$ and Remark \ref{Lp_loc} with $p=2$.
Denoting by $C$ any positive constant, independent of $u$, we get
%
%
\begin{equation}
\begin{aligned}
\label{annulus:cont:est:2}
	\int_{B_{R_2} \setminus B_{R_1}} K(|x|)|u|^{q_1} \, dx &\leq C \left(\int_{B_{R_2} \setminus B_{R_1}} |u|^2 \, dx \right)^{\frac{\tilde{q}-1}{2}} \|u\|^{1+q_1-\tilde{q}_s}  \\
	&\leq C \|u\|^{\tilde{q}-1}\|u\|^{1+q_1-\tilde{q}} 
	= C \|u\|^{q_1}.
\end{aligned}
\end{equation}
%
So, by \eqref{BR1:est} and \eqref{annulus:cont:est:2}, we obtain
\begin{equation}
\label{BR2:est}
	\int_{B_{R_2}} K(|x|)|u|^{q_1} \, dx \leq \|u\|^{q_1}\cS_0(q_1,R_1) + C\|u\|^{q_1} 
\end{equation}
and hence, recalling \eqref{BR2:comp:est}, we conclude that  $u \in L^{q_1}_K(B_{R_2}) \cap L^{q_2}_K(B^c_{R_2})$, and therefore $u \in L^{q_1}_K + L^{q_2}_K$. This implies that $H^s_{V,\rad}(\R^N) \subset L^{q_1}_K (\R^N) + L^{q_2}_K (\R^N)$. 

Moreover, if $\{u_n\}_n \subset H^s_{V,\rad}(\R^N)$ is a sequence such that $u_n \to 0$ in $H^s_{V,\rad}(\R^N)$, then, by \eqref{BR2:comp:est} and \eqref{BR2:est}, we have
\begin{align*}
	\int_{B_{R_2}} K(|x|)|u_n|^{q_1} \, dx + \int_{B^c_{R_2}} K(|x|)|u_n|^{q_2} \, dx \leq  C_3\|u_n\|^{q_1} + \|u_n\|^{q_2}\cS_{\infty}(q_2,R_2), 
\end{align*}
where the right-hand side goes to $0$ as $n \to +\infty$. By \cite[Proposition 2.7]{BaPiRo2011}, this yields that $u_n \to 0$ in $L^{q_1}_K + L^{q_2}_K$.
\end{proof}

\subsection{Compact embedding}
We now state and prove the main result of this Section, that is the compact embedding of $H^s_{V,\rad}(\R^N)$ into $L^{q_1}_K(\R^N) + L^{q_2}_K(\R^N)$.
\begin{Th}
\label{th:comp:emb}
Let $N \geq 2$, $s>1/2$, $q_1, q_2 > 1$ and $V$ and $K$ as in (V) and (K). If
\begin{equation}\tag{$\cS''$}
\label{S'':hyp}
	\lim_{R \to 0^+}\cS_0(q_1,R) = \lim_{R \to +\infty}\cS_{\infty}(q_2,R) = 0,
\end{equation}
then $H^s_{V,\rad}(\R^N)$ is compactly embedded into $L^{q_1}_K(\R^N) + L^{q_2}_K(\R^N)$.
\end{Th}
\begin{proof}
Let $\eps > 0$ and let $\{u_n\}_n \subset H^s_{V,\rad}(\R^N)$ be a sequence such that $u_n \weakto 0$ in $H^s_{V,\rad}(\R^N)$. It follows that $\left\{\|u_n\|\right\}_n$ is bounded and hence, as in \eqref{BR1:est} and \eqref{BR2:comp:est}, we can take $R_{\eps} > r_{\eps} > 0$ such that
\begin{equation}
\label{Br:est}
	\int_{B_{r_{\eps}}} K(|x|)|u_n|^{q_1} \, dx \leq \|u_n\|^{q_1}\cS_0(q_1,r_{\eps}) \leq \cS_0(q_1,r_{\eps}) \sup_n \|u_n\|^{q_1} < \frac{\eps}{3}
\end{equation}
and
\begin{equation}
\label{Br:comp:est}
	\int_{B^c_{R_{\eps}}} K(|x|)|u_n|^{q_2} \, dx \leq \|u_n\|^{q_2}\cS_{\infty}(q_2,R_{\eps}) \leq \cS_{\infty}(q_2,R_{\eps}) \sup_n \|u_n\|^{q_2} < \frac{\eps}{3}
\end{equation}
for all $n$ large enough.
Denoting by $C$ a positive constant independent of $u_n$, we use Lemma \ref{annulus:comp:emb} to get
%
%
\begin{equation}
\begin{aligned}
	\int_{B_{R_{\eps}} \setminus B_{r_{\eps}}} K(|x|)|u_n|^{q_1} \, dx &\leq C \left(\int_{B_{R_{\eps}} \setminus B_{r_{\eps}}} |u_n|^2 \, dx \right)^{\frac{\tilde{q}-1}{2}} \|u_n\|^{1+q_1-\tilde{q}} \\
	&\leq C \left(\int_{B_{R_{\eps}} \setminus B_{r_{\eps}}} |u_n|^2 \, dx \right)^{\frac{\tilde{q}-1}{2}}.
\end{aligned}
\end{equation}
By the compact embeddings of Remark \ref{cpt_loc}, this implies that  
\begin{equation} \label{annulus:comp:est}
	\int_{B_{R_{\eps}} \setminus B_{r_{\eps}}} K(|x|)|u_n|^{q_1} \, dx < \frac{\eps}{3}
\end{equation}
for all $n$ large enough. As a conclusion, \eqref{Br:est}, \eqref{Br:comp:est} and \eqref{annulus:comp:est} give
\[
\int_{B_{R_{\eps}}} K(|x|)|u_n|^{q_1} \, dx + \int_{B^c_{R_{\eps}}} K(|x|)|u_n|^{q_2} \, dx < \eps
\]
for $n$ sufficiently large, which means that $u_n \to 0$ in $L^{q_1}_K (\mathbb{R}^N)+ L^{q_2}_K (\mathbb{R}^N)$ by \cite[Proposition 2.7]{BaPiRo2011}.
\end{proof}

\begin{Rem} \normalfont
\label{rem:conditions}
Observe that condition \eqref{S'':hyp} implies condition \eqref{S':hyp}.
\end{Rem}

\section{How to get compactness}
\label{Operative:way:Sec}
Conditions \eqref{S':hyp} and \eqref{S'':hyp} are not so straightforward to be checked. In this section we provide some results with which it is easier to obtain the two conditions. Therefore, Theorem \ref{th:comp:emb} will be a combination of Theorem \ref{S0:holdness} and Theorem \ref{Sinf:holdness} below, which also provide intervals for the exponents $q_1$ and $q_2$ (see Section \ref{Examples:Sec} where we give some examples). 

We begin with a technical Lemma.
\begin{Lem}(see Lemma 2 in \cite{BaGuRo2015-1})
\label{est_1}
Let $\Omega \subset \R^N$ be a nonempty measurable set and
\[
	\Lambda:= \sup_{x \in \Omega}\dfrac{K(|x|)}{|x|^{\alpha}V(|x|)^{\beta}} < +\infty
\]
for some $0 \leq \beta \leq 1$, $\alpha \in \R$.
Let $u \in H^s_{V,\rad}(\R^N)$ and assume that there exist $\nu \in \R$ and $m>0$ such that 
\[
	|u(x)| \leq \dfrac{m}{|x|^{\nu}} \text{ a.e. on } \Omega.
\]
Then, for every $q > \max\{1,2\beta\}$ one has
\begin{multline*}
	\int_{\Omega} K(|x|)|u(x)|^q \, dx \\	
	\leq
	\begin{cases}
	 \Lambda m^{q-1} \left(\int_{\Omega} |x|^{\frac{2N}{N+2s(1-2\beta)}(\alpha - \nu(q-1))} \, dx  \right)^{\frac{N+2s(1-2\beta)}{2N}} S^{1-2\beta}\|u\|, &\mbox{ if } 0 \leq \beta \leq \frac12\\
 	 \Lambda  m^{q-2\beta}\left(\int_{\Omega} |x|^{\frac{1}{1-\beta}(\alpha-\nu(q-2\beta))} \,dx \right)^{1-\beta} \|u\|^{2\beta}, &\mbox{ if } \frac12 < \beta < 1\\
 	 \Lambda m^{q-2}\left(\int_{\Omega} |x|^{2(\alpha -\nu(q-2))}V(|x|)|u|^2 \, dx \right)^{\frac12} \|u\|, &\mbox{ if } \beta = 1
	\end{cases}
\end{multline*}
where $S$ is the Sobolev constant \eqref{Sobolev}.
\end{Lem}
\begin{proof}
We divide the proof in several cases, where we will often make use of the H\"older inequality with exponents $t$, $t_1$, $t_2$ which will be specified at the beginning of each case, together with their conjugates exponents.

\bigskip

\textit{Case $\beta=0$:} $t=2^*_s$ and $t'=(2^*_s)'=\frac{2N}{N+2s}$.
\begin{align*}
	\dfrac{1}{\Lambda}\int_{\Omega} K(|x|)|u|^{q-1}|u| \, dx &\leq \int_{\Omega} |x|^{\alpha}|u|^{q-1}|u| \, dx \\
	&\leq \left(\int_{\Omega}\left(|x|^{\alpha}|u|^{q-1}\right)^{(2_s^*)'} \, dx \right)^{\frac{1}{(2_s^*)'}} \left(\int_{\Omega}|u|^{2_s^*} \, dx \right)^{\frac{1}{2_s^*}} \\
	& \leq \left(\int_{\Omega} |x|^{\alpha\frac{2N}{N+2s}} |u|^{(q-1)\frac{2N}{N+2s}} \, dx\right)^{\frac{N+2s}{2N}} \|u\|_{L^{2_s^*}}\\
	&\leq m^{q-1}\left(\int_{\Omega} |x|^{\frac{2N}{N+2s}(\alpha-\nu(q-1))} \, dx\right)^{\frac{N+2s}{2N}} S [u]_{H^s} \\
	&\leq m^{q-1}\left(\int_{\Omega} |x|^{\frac{2N}{N+2s}(\alpha-\nu(q-1))} \, dx\right)^{\frac{N+2s}{2N}} S \|u\|.
\end{align*}

\bigskip

\textit{Case $0 < \beta < \frac12$:} $t_1=\frac{1}{\beta}$ with $t'_1=\frac{1}{1-\beta}$ and $t_2=\left(\frac{1-\beta}{1-2\beta}2^*_s\right)$ with $t'_2=\frac{2N(1-\beta)}{N+2s(1-2\beta)}$.
Writing $|u|=|u|^{2\beta}|u|^{1-2\beta}$, we have
\begin{align*}
	&\dfrac{1}{\Lambda}\int_{\Omega} K(|x|)|u|^{q-1}|u| \, dx \leq \int_{\Omega} |x|^{\alpha}|u|^{q-1}V(|x|)^{\beta}|u| \, dx \\
	& \leq \left( \int_{\Omega} \left(|x|^{\alpha}|u|^{q-1}|u|^{1-2\beta} \right)^{t'_1} \, dx \right)^{\frac{1}{t'_1}}\left(\int_{\Omega}\left(V(|x|)^{\beta}|u|^{2\beta}\right)^{t_1} \, dx\right)^{\frac{1}{t_1}}\\
	& =\left( \int_{\Omega} |x|^{\frac{\alpha}{1-\beta}}|u|^{\frac{q-1}{1-\beta}}|u|^{\frac{1-2\beta}{1-\beta}} \, dx \right)^{1-\beta}\left(\int_{\Omega} V(|x|)|u|^2 \, dx\right)^{\beta}\\
	& \leq  \left(\int_{\Omega} |x|^{\frac{\alpha}{1-\beta}}|u|^{\frac{q-1}{1-\beta}}|u|^{\frac{1-2\beta}{1-\beta}} \, dx  \right)^{1-\beta} \|u\|^{2\beta}\\
	& = \left[\left(\int_{\Omega} \left(|x|^{\frac{\alpha}{1-\beta}}|u|^{\frac{q-1}{1-\beta}}\right)^{t'_2} \, dx  \right)^{\frac{1}{t'_2}} \left(\int_{\Omega}\left(|u|^{\frac{1-2\beta}{1-\beta}}\right)^{t_2} \, dx  \right)^{\frac{1}{t_2}} \right]^{1-\beta}  \|u\|^{2\beta}\\
	&\leq \left[\left(\int_{\Omega} |x|^{\frac{\alpha}{1-\beta} t'_2}|u|^{\frac{q-1}{1-\beta}t'_2} \, dx  \right)^{\frac{1}{t'_2}} \|u\|^{\frac{1-2\beta}{1-\beta}}_{L^{2^*_s}} \right]^{1-\beta} \|u\|^{2\beta}\\
	& \leq \left[\left(\int_{\Omega} |x|^{\frac{\alpha}{1-\beta}t'_2}m^{\frac{q-1}{1-\beta}t'_2}|x|^{-\nu\left(\frac{q-1}{1-\beta}\right)t'_2} \, dx  \right)^{\frac{N+2s(1-2\beta)}{2N(1-\beta)}} S^{\frac{1-2\beta}{1-\beta}}\|u\|^{\frac{1-2\beta}{1-\beta}} \right]^{1-\beta} \|u\|^{2\beta}\\
	& \leq m^{q-1} \left(\int_{\Omega} |x|^{\frac{2N}{N+2s(1-2\beta)}(\alpha-\nu(q-1))} \, dx  \right)^{\frac{N+2s(1-2\beta)}{2N}}S^{1-2\beta}\|u\|.
\end{align*}

\bigskip

\textit{Case $\beta = \frac12$:} $t=t'=2$.
\begin{align*}
	\dfrac{1}{\Lambda}\int_{\Omega} K(|x|)|u|^{q-1}|u| \, dx &\leq \int_{\Omega} |x|^{\alpha}|u|^{q-1}V(|x|)^{\frac{1}{2}}|u| \, dx \\
	& \leq \left(\int_{\Omega} \left(|x|^{\alpha}|u|^{q-1}\right)^{t'} \, dx \right)^{\frac{1}{t'}}\left(\int_{\Omega} \left(V(|x|)^{\frac{1}{2}}|u|\right)^t \, dx \right)^{\frac{1}{t}} \\
	& = \left(\int_{\Omega} \left(|x|^{\alpha}|u|^{q-1}\right)^2 \, dx \right)^{\frac12}\left(\int_{\Omega} V(|x|)|u|^2 \, dx \right)^{\frac12} \\
	& \leq \left(\int_{\Omega} |x|^{2\alpha}|u|^{2(q-1)} \, dx \right)^{\frac12}\|u\| \\
	& \leq m^{q-1}\left(\int_{\Omega} |x|^{2\left(\alpha - \nu(p-1)\right)} \, dx \right)^{\frac12} \|u\|.
\end{align*}

\bigskip

\textit{Case $\frac12 < \beta < 1$:} $t_1=t'_1=2$ and $t_2=\frac{1}{2\beta-1}$ with $t'_2=\frac{1}{2-2\beta}$.

\begin{align*}
	&\dfrac{1}{\Lambda}\int_{\Omega} K(|x|)|u|^{q-1}|u| \, dx \leq \int_{\Omega} |x|^{\alpha}|u|^{q-1}V(|x|)^{\beta}|u| \, dx \\
	&\leq \left(\int_{\Omega} \left(|x|^{\alpha}|u|^{q-1}V(|x|)^{\beta-\frac{1}{2}}\right)^{t'_1} \, dx\right)^{\frac{1}{t'_1}}\left(\int_{\Omega} \left(V(|x|)^{\frac{1}{2}}|u|\right)^{t_1} \, dx \right)^{\frac{1}{t_1}} \\
	&= \left(\int_{\Omega} \left(|x|^{\alpha}|u|^{q-1}V(|x|)^{\beta-\frac{1}{2}}\right)^2 \, dx\right)^{\frac12}\left(\int_{\Omega} V(|x|)|u|^2 \, dx \right)^{\frac12} \\  
	&\leq \left(\int_{\Omega} |x|^{2\alpha}|u|^{2(q-1)} V(|x|)^{2\beta-1}\, dx\right)^{\frac12}\|u\| \\
	&\leq \left(\int_{\Omega} |x|^{2\alpha}|u|^{2(q-2\beta)} V(|x|)^{2\beta-1} |u|^{2(2\beta-1)} \, dx\right)^{\frac12}\|u\| \\
	&\leq \left[\left(\int_{\Omega} \left(|x|^{2\alpha}|u|^{2(q-2\beta)}\right)^{t'_2} \,dx \right)^{\frac{1}{t'_2}} \left(\int_{\Omega} \left(V(|x|)^{2\beta-1}|u|^{2(2\beta-1)}\right)^{t_2} \, dx \right)^{\frac{1}{t_2}}\right]^{\frac12}\|u\| \\
	&= \left[\left(\int_{\Omega} |x|^{\alpha\frac{1}{1-\beta}}|u|^{\frac{q-2\beta}{1-\beta}} \,dx \right)^{2(1-\beta)} \left(\int_{\Omega} V(|x|)|u|^2\, dx \right)^{2\beta-1}\right]^{\frac12}\|u\| \\
	&= \left(\int_{\Omega} |x|^{\alpha\frac{1}{1-\beta}}|u|^{\frac{q-2\beta}{1-\beta}} \,dx \right)^{1-\beta} \|u\|^{2\beta-1}\|u\| \\
	&\leq \left(\int_{\Omega} |x|^{\alpha\frac{1}{1-\beta}}m^{\frac{q-2\beta}{1-\beta}}|x|^{-\nu\frac{q-2\beta}{1-\beta}} \,dx \right)^{1-\beta} \|u\|^{2\beta} \\
	&= m^{q-2\beta}\left(\int_{\Omega} |x|^{\frac{1}{1-\beta}(\alpha-\nu(q-2\beta))} \,dx \right)^{1-\beta} \|u\|^{2\beta}.
\end{align*}

\bigskip

\textit{Case $\beta=1$:} $t=t'=2$.
\begin{align*}
	\dfrac{1}{\Lambda}\int_{\Omega} K(|x|)|u|^{q-1}|u| \, dx &\leq \int_{\Omega} |x|^{\alpha}|u|^{q-1}V(|x|)|u| \, dx \\
	&\leq \left(\int_{\Omega} \left(|x|^{\alpha}|u|^{q-1}V(|x|)^{\frac12}\right)^{t'} \, dx \right)^{\frac{1}{t'}} \left(\int_{\Omega} \left(V(|x|)^{\frac12}|u|\right)^t \, dx \right)^{\frac{1}{t}}\\
	&=\left(\int_{\Omega} \left(|x|^{\alpha}|u|^{q-1}V(|x|)^{\frac12}\right)^2 \, dx \right)^{\frac12} \left(\int_{\Omega} V(|x|)|u|^2 \, dx \right)^{\frac12}\\
	& \leq \left(\int_{\Omega} |x|^{2\alpha}|u|^{2(q-1)} V(|x|) \, dx \right)^{\frac12} \|u\|\\
	& = \left(\int_{\Omega} |x|^{2\alpha}|u|^{2(q-2)} V(|x|)|u|^2 \, dx \right)^{\frac12} \|u\|\\
	& \leq m^{q-2}\left(\int_{\Omega} |x|^{2(\alpha - \nu(q-2))}V(|x|)|u|^2 \, dx \right)^{\frac12} \|u\|.
\end{align*}
\end{proof}

We introduce the functions $\alpha^*_s \colon [0,1] \to \mathbb{R}$ defined by 
\begin{equation}
\label{alpha_star_def}
	\alpha^*_s(\beta)=
	\begin{dcases}
		-\frac{N}{2}-(1-2\beta)s, &\mbox{ if }  0 \leq \beta \leq \frac12\\
		-(1-\beta)N, &\mbox{ if } \frac12 \leq \beta \leq 1 
	\end{dcases}
\end{equation}
and $q^*\colon \R \times [0,1] \times (0,1) \to \R$ defined as
\begin{equation}
\label{q_up_star_def}
	q^*(\alpha, \beta, s) =2\frac{\alpha - 2s\beta + N}{N-2s}.
\end{equation}


\begin{Rem}
\label{rem:alpha*:q*}
\begin{enumerate}[label=(\roman*)]
	\item $\alpha^*_s(1) = 0$ and $\alpha^*_s(\beta)<0$ for every $0\leq \beta < 1$.
	\item The function $q^*(\alpha, \beta, s)$ is increasing in $\alpha$ and decreasing in $\beta$.
\end{enumerate}
\end{Rem}

\subsection{Holdness of condition on \texorpdfstring{$\cS_0$}{S0}}
\begin{Th}
\label{S0:holdness}
Let $N \geq 2$, $s >1/2$ and let $V$, $K$ satisfy $(V)$ and $(K)$. Assume that there exists $R_1 > 0$ such that
\begin{equation}
\label{esssup:S0}
	\sup_{r \in (0,R_1)} \dfrac{K(r)}{r^{\alpha_0}V(r)^{\beta_0}} < +\infty
\end{equation}
for some $0 \leq \beta_0 \leq 1$ and $\alpha_0 > \alpha^*_s(\beta_0)$.

Then
\[
	\lim_{R \to 0^+} \cS_0(q_1,R)=0
\]
for every $q_1 \in \R$ such that
\[
	\max\{1,2\beta_0 \} < q_1 < q^*(\alpha_0,\beta_0,s).
\]
\end{Th}

\begin{proof}
Let $u \in H^s_{V,\rad}$ such that $\|u\|=1$. Let $0 \leq R \leq R_1$, using Theorem \ref{Strauss_Th}, then we can apply Lemma \ref{est_1} with: 
\[
	\Omega=B_R, \quad \alpha=\alpha_0, \quad \beta=\beta_0, \quad m=C(N,s)\|u\|=C(N,s), \quad \nu=\frac{N-2s}{2}.
\]

\textit{Case $0 \leq \beta_0 \leq \frac12$:}
\begin{multline*}
	\int_{B_R} K(|x|)|u|^{q_1} \, dx \\
	\leq \Lambda C(N,s)^{q_1-1}\left(\int_{B_R}|x|^{\frac{2N}{N+2s(1-2\beta_0)}\left(\alpha_0-\frac{N-2s}{2}(q_1-1)\right)} \, dx \right)^{\frac{N+2s(1-2\beta_0)}{2N}} S^{1-2\beta_0} \\
	= C \left(\int_0^R r^{\frac{2N}{N+2s(1-2\beta_0)}\left(\alpha_0-\frac{N-2s}{2}(q_1-1)\right)+N-1} \, dr \right)^{\frac{N+2s(1-2\beta_0)}{2N}}\\
	=  \wt{C} \left(R^{\frac{2N}{N+2s(1-2\beta_0)}\left(\alpha_0-\frac{N-2s}{2}(q_1-1)\right)+N}\right)^{\frac{N+2s(1-2\beta_0)}{2N}}
\end{multline*}
where
\begin{multline*}
	\frac{2N}{N+2s(1-2\beta_0)}\left(\alpha_0-\frac{N-2s}{2}(q_1-1)\right)+N\\
	=\frac{N-2s}{N+2s(1-2\beta_0)}\left(\frac{2\alpha_0 - 4s\beta_0 + 2N }{N-2s} - q_1 \right) = \frac{N-2s}{N+2s(1-2\beta_0)}(q^*(\alpha_0,\beta_0,s) - q_1) > 0.
\end{multline*}

\bigskip

\textit{Case $\frac12 < \beta_0 < 1$:}
\begin{multline*}
	\int_{B_R} K(|x|)|u(x)|^{q_1} \, dx \\ 
	\leq \Lambda C(N,s)^{q_1-2\beta_0}\left(\int_{B_R}|x|^{\frac{1}{1-\beta_0}\left(\alpha_0-\frac{N-2s}{2}(q_1-2\beta_0)\right)} \, dx  \right)^{1-\beta_0}\\
	= C \left(\int_0^R r^{\frac{1}{1-\beta_0}\left(\alpha_0-\frac{N-2s}{2}(q_1-2\beta_0)\right) + N - 1} \, dr \right)^{1-\beta_0} \\
	= \wt{C} \left(R^{\frac{1}{1-\beta_0}\left(\alpha_0-\frac{N-2s}{2}(q_1-2\beta_0)\right) + N}  \right)^{1-\beta_0}
\end{multline*}
where
\begin{multline*}
	\frac{1}{1-\beta_0}\left(\alpha_0-\frac{N-2s}{2}(q-2\beta_0)\right) + N =\frac{N-2s}{2(1-\beta_0)}\left(\frac{2\alpha_0 - 4\beta_0 s + 2N}{N-2s} - q_1 \right) \\
	=\frac{N-2s}{2(1-\beta_0)}\left(q^*\left(\alpha_0, \beta_0, s\right)-q_1\right) > 0.
\end{multline*}

\bigskip

\textit{Case $\beta_0 = 1$:}
\begin{align*}
	\int_{B_R} K(|x|)|u(x)|^{q_1} \, dx & \leq \Lambda C(N,s)^{q_1-2}\left(\int_{B_R} |x|^{2\left(\alpha_0-\frac{N-2s}{2}(q_1-2)\right)}V(|x|)|u|^2 \, dx \right)^{\frac12} \\
	& \leq C\left(R^{2\left(\alpha_0-\frac{N-2s}{2}(q_1-2)\right)}\int_{B_R} V(|x|)|u|^2 \, dx \right)^{\frac12}\\
	& \leq \wt{C}R^{\frac{2\alpha_0-(N-2s)(q_1-2)}{2}}
\end{align*}
where
\begin{multline*}
	\frac{2\alpha_0-(N-2s)(q_1-2)}{2}
	=\frac{2\alpha_0 + 2N - 4s+(N-2s)q_1}{2}  \\
	= \frac{N-2s}{2}\left(q^*\left(\alpha_0,1,s\right)-q_1\right) > 0.
\end{multline*}

We define the positive function
\[
	\delta_0:=\delta_0(N,\alpha_0,\beta_0,q_1,s)=
	\begin{dcases}
		\frac{N-2s}{N+2s(1-2\beta_0)}\left(q^*\left(\alpha_0,\beta_0,s\right) - q_1 \right), &\mbox{ if } 0 \leq \beta_0 \leq \frac12; \\
		\frac{N-2s}{2(1-\beta_0)}\left(q^*\left(\alpha_0,\beta_0,s\right) - q_1 \right), &\mbox{ if } \frac12 < \beta_0 < 1; \\
		\frac{N-2s}{2}\left(q^*\left(\alpha_0,\beta_0,s\right) -q_1 \right), &\mbox{ if } \beta_0 = 1.
	\end{dcases}
\]
and we observe that, for every $0 \leq \beta_0 \leq 1$ we have
\[
	\cS_0(q_1,R) \leq CR^{\delta_0}.
\]
To finish the proof, it is enough to take the limit as $R \to 0^+$.
\end{proof}

\begin{Rem} \normalfont
\label{Rem_q_up_star}
The request $\alpha_0 > \alpha^*_s(\beta_0)$ is equivalent to $q^*(\alpha_0,\beta_0,s) > \max\left\{1,2\beta_0\right\}$. In fact, recalling \eqref{alpha_star_def} and \eqref{q_up_star_def}, $q^*(\alpha_0,\beta_0,s) > 1$ holds if and only if
	$\alpha_0 > -N/2 - (1-2\beta_0)s$,
while $q^*(\alpha_0,\beta_0,s) > 2\beta_0$ holds if and only if
$\alpha_0 > -(1-\beta_0)N$.
Therefore, no further assumptions are needed on $\alpha_0$ and $\beta_0$.
\end{Rem}

\subsection{Holdness of condition on \texorpdfstring{$\cS_{\infty}$}{Sinf}}
\begin{Th}
\label{Sinf:holdness}
Let $N \geq 2$, $s > 1/2$ and let $V$, $K$ satisfy $(V)$ and $(K)$. Assume that there exists $R_2 > 0$ such that
\begin{equation}
\label{esssup:Sinf}
	\sup_{r > R_2} \dfrac{K(r)}{r^{\alpha_{\infty}}V(r)^{\beta_{\infty}}} < +\infty
\end{equation}
for some $0 \leq \beta_{\infty} \leq 1$ and $\alpha_{\infty} \in \R$.

Then
\[
	\lim_{R \to +\infty} \cS_{\infty}(q_2,R)=0
\] 
for every $q_2 \in \R$ such that
\[
	q_2 > \max\left\{1,2\beta_{\infty},q^*\left({\alpha_{\infty},\beta_{\infty},s}\right)\right\}.
\]
\end{Th}

\begin{proof}
Let $u \in H^s_{V,\rad}$ such that $\|u\|=1$. Let $R \geq R_2$, using Theorem \ref{Strauss_Th}, then we can apply Lemma \ref{est_1} with: 
\[
	\Omega=B_R^c, \quad \alpha=\alpha_{\infty}, \quad \beta=\beta_{\infty}, \quad m=C(N,s)\|u\|=C(N,s), \quad \nu=\frac{N-2s}{2}.
\]

\textit{Case $0 \leq \beta_{\infty} \leq \frac12$:}
\begin{multline*}
	\int_{B_R^c} K(|x|)|u|^{q_2} \, dx \\
	\leq \Lambda C(N,s)^{q_2-1}\left(\int_{B_R}|x|^{\frac{2N}{N+2s(1-2\beta_{\infty})}\left(\alpha_{\infty}-\frac{N-2s}{2}(q_2-1)\right)} \, dx \right)^{\frac{N+2s(1-2\beta_{\infty})}{2N}} S^{1-2\beta_{\infty}} \\
	= C \left(\int_R^{+\infty} r^{\frac{2N}{N+2s(1-2\beta_{\infty})}\left(\alpha_{\infty}-\frac{N-2s}{2}(q_2-1)\right)+N-1} \, dr \right)^{\frac{N+2s(1-2\beta_{\infty})}{2N}}\\
	=  \wt{C} \left(R^{\frac{2N}{N+2s(1-2\beta_{\infty})}\left(\alpha_{\infty}-\frac{N-2s}{2}(q_2-1)\right)+N}\right)^{\frac{N+2s(1-2\beta_{\infty})}{2N}}
\end{multline*}
where
\begin{multline*}
	\frac{2N}{N+2s(1-2\beta_{\infty})}\left(\alpha_{\infty}-\frac{N-2s}{2}(q_2-1)\right)+N\\
	=\frac{N(N-2s)}{N+2s(1-2\beta_\infty)}\left(\frac{2\alpha_{\infty} - 4s\beta_{\infty} + 2N}{N-2s} - q_2 \right) \\
	= \frac{N(N-2s)}{N+2s(1-2\beta_\infty)}(q^*\left(\alpha_{\infty},\beta_{\infty},s) - q_2\right) < 0.
\end{multline*}

\bigskip

\textit{Case $\frac12 < \beta_{\infty} < 1$:}
\begin{align*}
	\int_{B_R^c} K(|x|)|u(x)|^{q_2} \, dx &\leq \Lambda C(N,s)^{q_2-2\beta_0}\left(\int_{B_R^c}|x|^{\frac{1}{1-\beta_{\infty}}\left(\alpha_{\infty}-\frac{N-2s}{2}(q_2-2\beta_{\infty})\right)} \, dx  \right)^{1-\beta_{\infty}}\\
	&= C \left(\int_R^{+\infty} r^{\frac{1}{1-\beta_{\infty}}\left(\alpha_{\infty}-\frac{N-2s}{2}(q_2-2\beta_{\infty})\right) + N - 1} \, dr \right)^{1-\beta_{\infty}} \\
	&= \wt{C} \left(R^{\frac{1}{1-\beta_{\infty}}\left(\alpha_{\infty}-\frac{N-2s}{2}(q_2-2\beta_{\infty})\right) + N}  \right)^{1-\beta_{\infty}}
\end{align*}
where
\begin{align*}
	\frac{1}{1-\beta_{\infty}}\left(\alpha_{\infty}-\frac{N-2s}{2}(q_2-2\beta_{\infty})\right) + N &=\frac{N-2s}{2(1-\beta_{\infty})}\left(\frac{2\alpha_{\infty} - 4\beta_{\infty} s + 2N}{N-2s} - q_2 \right) \\
	&=\frac{N-2s}{2(1-\beta_{\infty})}\left(q^*\left(\alpha_{\infty}, \beta_{\infty}, s\right)-q_2\right) < 0.
\end{align*}

\bigskip

\textit{Case $\beta_{\infty} = 1$:}
\begin{align*}
	\int_{B_R^c} K(|x|)|u(x)|^{q_2} \, dx &\leq \Lambda C(N,s)^{q_2-2}\left(\int_{B_R^c} |x|^{2\left(\alpha_{\infty}-\frac{N-2s}{2}(q_2-2)\right)}V(|x|)|u|^2 , dx \right)^{\frac12} \\
	& \leq C\left(R^{2\left(\alpha_{\infty}-\frac{N-2s}{2}(q_2-2)\right)}\int_{B_R^c} V(|x|)|u|^2 \, dx \right)^{\frac12}\\
	& \leq \wt{C}R^{\frac{2\alpha_{\infty}-(N-2s)(q_2-2)}{2}}
\end{align*}
where
\begin{multline*}
	\frac{2\alpha_{\infty}-(N-2s)(q_2-2)}{2}=\frac{N-2s}{2}\left(\frac{2\alpha_{\infty} + 2N - 4s}{N-2s} - q_2 \right) \\
	= \frac{N-2s}{2}\left(q^*\left(\alpha_{\infty},1,s\right)-q_2\right) < 0.
\end{multline*}

We define the negative function
\[
	\delta_{\infty}=\delta(N,\alpha_{\infty},\beta_{\infty},q_2,s)=
	\begin{dcases}
		\frac{N(N-2s)}{N+2s(1-2\beta_\infty)}\left(q^*\left(\alpha_{\infty}, \beta_{\infty}, s\right)- q_2 \right), &\mbox{ if } 0 \leq \beta_{\infty} \leq \frac12; \\
		\frac{N-2s}{2(1-\beta_{\infty})}\left(q^*\left(\alpha_{\infty}, \beta_{\infty}, s\right) - q_2 \right), &\mbox{ if } \frac12 < \beta_{\infty} < 1; \\
		\frac{N-2s}{2}\left(q^*\left(\alpha_{\infty}, \beta_{\infty}, s\right) - q_2 \right), &\mbox{ if } \beta_{\infty} = 1.
	\end{dcases}
\]
and we observe that, for all $0 \leq \beta_{\infty} \leq 1$ we have
\[
	\cS_{\infty}(q_1,R) \leq CR^{\delta_{\infty}}.
\]
We conclude by taking the limit as  $R \to +{\infty}$.
\end{proof}

\begin{Rem} \normalfont
By Remark \ref{Rem_q_up_star}, we observe that
\[
	\max\left\{1,2\beta,q^*(\alpha,\beta,s) \right\} = 
	\begin{dcases}
		q^*(\alpha,\beta,s), &\text{if } \alpha > \alpha^*_s(\beta),\\
		\max\left\{1,2\beta\right\}, &\text{otherwise}.
	\end{dcases}
\]
\end{Rem}

\section{Existence and multiplicity}
\label{Existence:Sec}
\label{Existence:section}
Using the compact embedding of Theorem \ref{th:comp:emb}, we are now able to show the existence of solutions to equation \eqref{1}, that is,
\[
	(-\Delta)^su + V(|x|)u = K(|x|)f(u) \quad \text{ in } \R^N,
\]
where $N \geq 2$, $V$ and $K$ are such that (V) and (K) hold, and  $f\colon \R \to \R$ is continuous and, together with $F(s):=\int_0^s f(\tau) \, d\tau$, it satisfies the following assumptions:
\begin{enumerate}[label=\textbf{(f\arabic*)}]
	\item (double-power growth condition) there exist $C>0$ and $q_1>2$, $q_2>2$ such that
		\[
			|f(t)| \leq C\min\left\{|t|^{q_1-1},|t|^{q_2-1} \right\} \text{ for all } t \in \R;
		\]
	\item there exists $t_0 > 0$ such that $F(t_0)>0$;
	\item (Ambrosetti-Rabinowitz condition) there exists $\mu>2$ such that
		\[
			0 \leq \mu F(s) \leq f(s)s \text{ for all } s \in \R.
		\]
\end{enumerate}
Some remarks on the hypotheses are in order.
\begin{Rem}\label{rem:hyp:f}
\begin{enumerate}
	\item Condition \textbf{(f1)} implies that there exists a positive constant $\wt{C}>0$ such 
		\[
			|f(t)| \leq \wt{C}|t|^{q-1}
		\]
		for $q=q_1$, $q=q_2$ and in particular for every $q$ in between.
	\item The assumptions $q_1 \neq q_2$ is never asked: in fact, we allow the case $q_1=q_2$ (i.e. the single-power growth condition).
	\item The \textit{Ambrosetti-Rabinowitz condition} implies that $q_1,q_2 \geq \mu$. In fact, \textbf{(f3)} implies that there exists $R>0$ and $C(R)>0$ such that
		\[
			F(t) \geq \frac{|t|^{\mu}}{R^{\mu}} - C(R),
		\]
	that is $F$ has superquadratic growth (or similarly $f$ has superlinear growth), so it follows that $q_1,q_2 \geq \mu$.
\end{enumerate}
\end{Rem}

\begin{Ex}
	The simplest nonlinearity satisfying the above assumption is 
	\[g\left( t\right) =\min \left\{ \left| t\right| ^{q_{1}-2}t,\left| t\right|
	^{q_{2}-2}t\right\} ,\qquad q_{1},q_{2}>4\,,
	\]
	which ensures \textbf{(f2)} with $\mu = \min \left\{ \frac{q_1}{2}, \frac{q_2}{2} \right\} $. Another model example is 
	\[
	g\left( t\right) =\frac{\left| t\right| ^{q_{2}-2}t}{1+\left| t\right|
		^{q_{2}-q_{1}}}\quad \text{with } 4<q_{1}\leq q_{2}\,,
	\]
	for which \textbf{(f2)} holds with $\mu =\frac{q_1}{2}$. Note that, in
	both these cases, $g$
	becomes $g\left( t\right) =\left| t\right| ^{q-2}t$ if $q_{1}=q_{2}=q$. 
\end{Ex}

Our approach is variational: we will find solutions to \eqref{1} as critical points of the  energy functional $\cE\colon H^s_{V,\rad} \to \R$ associated to the equation, namely
\begin{align*}
	\cE(u)&=\frac{C(N,s)}{4}\int_{\R^N \times \R^N} \frac{|u(x)-u(y)|^2}{|x-y|^{N+2s}} \, dx \, dy + \frac12\int_{\R^N}V(|x|)|u|^2 \, dx \\
	& \quad - \int_{\R^N}K(|x|)F(u) \, dx \\
	& =\frac12\|u\|^2 - \int_{\R^N}K(|x|)F(u) \, dx.
\end{align*}

By Proposition 3.8 in \cite{BaPiRo2011} it follows that, if \eqref{S':hyp} holds, the functional $\cE$ is of class $C^1$ on $H^s_{V,\rad}(\R^N)$
and its  G\^{a}teaux derivative at $u \in H^s_{V,\rad}(\R^N)$ along $v \in H^s_{V,\rad}(\R^N)$ is given by
\begin{multline*}
	\cE'(u)[v] = \frac{C(N,s)}{2}\int_{\R^N \times \R^N} \frac{|u(x)-u(y)||v(x)-v(y)|}{|x-y|^{N+2s}} \, dx \, dy \\
	+ \frac12\int_{\R^N}V(|x|)uv \, dx - \int_{\R^N}K(|x|)f(u)v \, dx.
\end{multline*}

Hence $u \in H^s_{V,\rad}(\R^N)$ is a \textit{weak solution} to \eqref{1} if and only if $\cE'(u)=0$, that is if and only if $u$ is a critical point for $\cE$.

The following Lemma gives a lower bound estimate for the energy functional.
\begin{Lem}
\label{Lemma:lower:bound}
If $s>1/2$ and \eqref{S'':hyp} holds, then there exist two constants $C_1>0$, $C_2>0$ (both independent of $u$), such that
\begin{equation}
\label{en:funct:lower:bound}
	\cE(u) \geq \frac12\|u\|^2 - C_1\|u\|^{q_1} - C_2\|u\|^{q_2}
\end{equation}
for all $u \in H^s_{V,\rad}(\R^N)$.
\end{Lem}
\begin{proof}
Without loss of generality, we can assume that $R_1 < R_2$ in condition \eqref{S':hyp} (see Remark \ref{rem:conditions}). We have
\begin{align*}
	&\left|\int_{\R^N}K(|x|)F(u) \, dx \right| \leq C\int_{\R^N} K(|x|)\min\{|u|^{q_1},|u|^{q_2}\} \, dx \\
	& = C\left(\int_{B_{R_2}}K(|x|)|u|^{q_1} \, dx + \int_{B^c_{R_2}} K(|x|)|u|^{q_2} \, dx \right) \\
	& = C\left(\int_{B_{R_1}}K(|x|)|u|^{q_1} \, dx + \int_{B_{R_2} \setminus B_{R_1}}K(|x|)|u|^{q_1} \, dx + \int_{B^c_{R_2}} K(|x|)|u|^{q_2} \, dx \right)
\end{align*}
and we want to estimate all the three integrals on the last line. For the first and the third we recall \eqref{BR1:est} and \eqref{BR2:comp:est}. Regarding the middle integral, from \eqref{ann:estimate} we have that there exists a constant $\bar{C}:=\bar{C}(N,R_1,R_2,q,s,t)>0$ such that
\[
	\int_{B_{R_2} \setminus B_{R_1}}K(|x|)|u|^{q_1} \, dx \leq \bar{C}\|u\|^l\|u\|^{q_1-l}_{L^2(B_{R_2} \setminus B_{R_1})} \leq \bar{C}\|u\|^{q_1}.
\]
Hence, from the above estimate, \eqref{BR1:est} and \eqref{BR2:comp:est} we get
\begin{align*}
	\left|\int_{\R^N}K(|x|)F(u) \, dx \right| &\leq C\left(\|u\|^{q_1}\cS_0(q_1,R_1) + \bar{C}\|u\|^{q_1} + \|u\|^{q_2}\cS_{\infty}(q_2,R_2) \right) \\
	&=C_1\|u\|^{q_1} + C_2\|u\|^{q_2}.
\end{align*}
\end{proof}

In the next Lemma, we prove that the Palais-Smale condition holds for the functional $\cE$.
\begin{Lem}
\label{PS:holdness}
If $s>1/2$ and \eqref{S'':hyp} holds, then $\cE\colon H^s_{V,\rad}(\R^N) \to \R$ satisfies the $(PS)$-condition.
\end{Lem}
\begin{proof}
Let $\{ u_n\}_n \in H^s_{V,\rad}(\R^N)$ be a $(PS)$-sequence for $\cE$, that is a sequence such that
\begin{equation}
\label{PS:seq}
	\left\{ \cE(u_n) \right\}_n \text{ is bounded and } \cE'(u_n) \to 0 \text{ in the dual space of } H^s_{V,\rad}(\R^N).
\end{equation}
This means that
\[
	\cE(u_n) = \frac12\|u_n\|^2 - \int_{\R^N}K(|x|)F(u_n) \, dx = O(1)
\]
and
\[
	\cE'(u_n)[u_n] = \|u_n\|^2 - \int_{\R^N}K(|x|)f(u_n)u_n \, dx = o(1)\|u_n\|.
\]
Since $f$ satisfies \textbf{(f3)} we have that
\begin{align*}
	\frac12\|u_n\|^2 + O(1) &= \int_{\R^N}K(|x|)F(u_n) \, dx \\
	&\leq \frac{1}{\mu}\int_{\R^N}K(|x|)f(u_n)u_n \, dx = \frac{1}{\mu} \|u_n\|^2 + o(1)\|u_n\|
\end{align*}
hence
\[
	\left(\frac12 - \frac{1}{\mu}\right)\|u_n\|^2 \leq o(1) \|u_n\|
\]
and we get that $\{\|u_n\|\}_n$ is bounded since $\mu > 2$. Therefore, there exists $u \in H^s_{V,\rad}(\R^N)$ such that \
\begin{equation}
\label{weak:conv}
	u_n \weakto u \text{ in } H^s_{V,\rad}(\R^N).
\end{equation} 

Thanks to \eqref{S'':hyp} we can apply Thereom \ref{th:comp:emb} obtaining
\[
	u_n \to u \in L^{q_1}_K(\mathbb{R}^N) + L^{q_2}_K(\mathbb{R}^N).
\]
Now, from \eqref{PS:seq} we have that
\[
	\|u_n\|^2 = \cE'(u_n)[u_n] + \int_{\R^N}K(|x|)f(u_n)u_n \, dx
\]
and by \cite{BaPiRo2011}, Proposition 3.8, this converges to
\[
	\|u\|^2 = o(1) + \int_{\R^N}K(|x|)f(u)u \, dx
\]
as $n \to +\infty$. Hence, $\lim_{n \to +\infty}\|u_n\|^2$ exists and, by the weak lower semicontinuity of the norm, we obtain
\begin{equation}
\label{A}
	\|u\|^2 \leq \lim_{n \to +\infty}\|u_n\|^2.
\end{equation}
Exploting the convexity of the norm, we have
\begin{align*}
	\frac12\|u\|^2 - \frac12\|u_n\|^2 &\geq \|u_n\|^2(u-u_n) \\
	&= \cE'(u_n)[u-u_n] + \int_{\R^N}K(|x|)f(u_n)(u-u_n) \, dx = o(1),
\end{align*}
and it follows that
\[
	\frac12\|u\|^2 \geq \frac12\|u_n\|^2 + o(1).
\]
Passing to the limit for $n \to +\infty$ leads to
\begin{equation}
\label{B}
	\frac12\|u\|^2 \geq \frac12\lim_{n \to +\infty}\|u\|^2. 
\end{equation}

From \eqref{A} and \eqref{B} we obtain
\[
	\frac12\|u\|^2	 \leq \frac12\lim_{n \to +\infty}\|u_n\|^2 \leq \frac12\|u\|^2,
\]
that is $\|u_n\| \to \|u\|$, that together with \eqref{weak:conv} gives that $u_n \to u$ in $H^s_{V,\rad}(\R^N)$.
\end{proof}

Now, we state and prove our main existence result, which provides the existence of a nonnegative nontrivial solution to \eqref{1}. The proof relies on the Mountain-Pass Theorem.
\begin{Th}
\label{exist:th}
Suppose that $s>1/2$ and \eqref{S'':hyp} holds, together with $\textbf{(f1)}, \textbf{(f2)}, \textbf{(f3)}$. Then equation \eqref{1} admits a solution $u \in H^s_{V,\rad}(\R^N)$, $u \neq 0$, $u \geq 0$.
\end{Th}
\begin{proof}
From \eqref{en:funct:lower:bound} we have that
\[
	\cE(u) \geq \frac12\|u\|^2 - C_1\|u\|^{q_1} - C_2\|u\|^{q_2} = \|u\|^2\left(\frac12- C_1\|u\|^{q_1-2} - C_2\|u\|^{q_2-2} \right)
\]
and considering the infimum over the functions $u \in H^s_{V,\rad}(\R^N)$ such that $\|u\|=\rho$, for a fixed $\rho>0$, we get
\[
	\inf_{\substack{u \in H^s_{V,\rad}(\R^N) \\ \|u\|=\rho}} \cE(u) \geq \rho^2\left(\frac12- C_1\rho^{q_1-2} - C_2\rho^{q_2-2} \right). 
\]
Hence, choosing $\rho>0$ such that $\frac12- C_1\rho^{q_1-2} - C_2\rho^{q_2-2} > 0$, we obtain
\begin{equation}
\label{MP:geom:i}
	\inf_{\substack{u \in H^s_{V,\rad}(\R^N) \\ \|u\|=\rho}} \cE(u) > 0.
\end{equation}

By Lemma \ref{PS:holdness} it follows that the Palais-Smale condition holds, so it only remains to prove that there exists an element $e \in H^s_{V,\rad}(\R^N) \setminus \bar{B}_{\rho}$ such that
\[
	\cE(e) < 0.
\]

Fix a $t_0 \in (0,t)$, integrating \textbf{(f3)}  from $t_0$ to $t$ we obtain
\begin{equation}
\label{F:est}
	F(t) \geq \frac{F(t)}{t_0^{\mu}}t^{\mu}
\end{equation}
for all $t \geq t_0$.
Thanks to the density of $C^{\infty}_c(\R^N)$ in $\dot{H}^s(\R^N)$, we fix a function $u_0 \in C^{\infty}_c(\R^N \setminus \{0\})$, $u_0 \geq 0$, such that the set
\[
	\left\{x \in \R^N : u_0(x) \geq t_0 \right\}
\]
has positive Lebesgue measure. Fixing also a $\lambda > 1$, using \eqref{F:est}, we have
\begin{align*}
	\int_{\R^N} K(|x|) F(\lambda u_0) \, dx &\geq \int_{\{u_0 \geq t_0\}} K(|x|)F(\lambda u_0) \, dx \geq \frac{\lambda^{\mu}}{t_0^{\mu}} \int_{\{u_0 \geq t_0\}} K(|x|)F(t_0)u_0^{\mu} \, dx \\
	& \geq \lambda^{\mu}\int_{\{u_0 \geq t_0\}} K(|x|)F(t_0) \, dx
\end{align*}
and this is positive thanks to \textbf{(f2)}.
Hence, since $\mu > 2$,
\[
	\lim_{\lambda \to +\infty} \cE(\lambda u_0) \leq \lim_{\lambda \to +\infty} \left(\frac{\lambda^2}{2}\|u_0\|^2 - \lambda^{\mu}\int_{\{u_0 \geq t_0\}} K(|x|)F(t_0) \right) = -\infty.
\]
Setting $e:=\lambda u_0$, for $\lambda$ sufficiently large, we proved that also the last geometric hypothesis of the Mountain Pass holds.
Hence, there exists a critical point $u \in H^s_{V,\rad}(\R^N)$, $u \neq 0$, for $\cE$. Up to suitably modifying $f(t)$ for $t<0$ (for instance setting $f(t)=0$), it is a standard exercise to conclude that $u$ is nonnegative.
\end{proof}

{Supposing that the nonlinearity enjoys some symmetry property, then we are able to prove a multiplicity existence result.
\begin{Th}\label{multi:th}
Under the same hypotheses of Theorem \ref{exist:th}, if moreover $f$ is odd, then there exist infinitely many solutions to equation \eqref{1}.
\end{Th}
\begin{proof}
We want to use the $\Z^2-$version of the Mountain-Pass Theorem with $X_1=\{0\}$ and $X_2=H^s_{V,\rad}(\R^N)$. Since Lemma \ref{PS:holdness} is still true if $f$ is odd, we only need to prove the following:
\begin{enumerate}[label=\roman*)]
	\item there exist $\rho,\alpha>0$ such that $\inf(S_{\rho} \cap H^s_{V,\rad}(\R^N)) \geq \alpha$;
	\item for every finite dimensional $Y \subset H^s_{V,\rad}(\R^N)$ there exists $R=R(Y)>0$ such that
	\[
		\cE(u) < 0
	\]
	for every $u \in Y \setminus \overline{B}_{\rho}$.
\end{enumerate}
Condition i) directly follows from \eqref{MP:geom:i}. To get condition ii), let $Y \subset H^s_{V,\rad}(\R^N)$ be a finite dimensional subspace and let $(u_n)_n \subset Y$ be a sequence such that $\|u_n\| \to +\infty$ as $n \to +\infty$. By Lemma \ref{Lemma:lower:bound} there exists two constant $C_1,C_2>0$ such that 
\[
	\cE(u_n) \leq \frac12\|u_n\|^2 - C_1\|u_n\|^{q_1} - C_2\|u_n\|^{q_2} \to -\infty
\]
since $q_1,q_2>2$ (see Remark \ref{rem:hyp:f}).

Hence, there exists an unbounded sequence of critical values of $\cE$.
\end{proof}

\section{Examples}
\label{Examples:Sec}
In this section we collect some examples which show how to apply our results to get compact embeddings for different potentials, and thus solutions to equation \eqref{1} by Theorems \ref{exist:th} and \ref{multi:th}. We assume as usual $N\geq 2$ and $s>\frac{1}{2}$.

\begin{Ex}[power type potentials]
 Consider the potentials
\[
	V(r)=r^a \quad \text{and} \quad K(r)=r^b, \quad a \in \R,\quad b > -\frac{N}{2}-s.
\]
Then
\[
	\sup_{r \in (0,R_1)} \frac{K(r)}{r^{\alpha_0}V(r)^{\beta_0}} = \sup_{r \in (0,R_1)} \frac{1}{r^{\alpha_0+a\beta_0 - b}},
\]
and assumption \eqref{esssup:S0} holds true if and only if $\beta_0\in[0,1]$ and $\alpha_0 \leq b - a\beta_0$. 
On the other hand, assumption \eqref{esssup:Sinf} holds if and only if $\beta_{\infty}\in[0,1]$ and $\alpha_{\infty} \geq b - a\beta_{\infty}$. 
By Remark \ref{rem:alpha*:q*}, we choose $\beta_0=0$, $\beta_{\infty}=1$ and thus $\alpha_0=b$, $ \alpha_{\infty}=b-a$. Hence, by Theorems \ref{S0:holdness} and \ref{Sinf:holdness}, we get that
\[
	\lim_{R_1 \to 0^+} \cS_0(q_1,R_1) = 0
\]
for every $q_1 \in \R$ such that 
\[
	1 < q_1 < q^*(\alpha_0,\beta_0,s)=2\frac{b + N }{N-2s}=2\left(1+\frac{b + 2s }{N-2s}\right) 
\]
(note that $\alpha_0=b>\alpha^*_s(\beta_0)=\alpha^*_s(0)=-N/2-s$ by assumption) and
\[
	\lim_{R_2 \to +\infty} \cS_{\infty}(q_2,R_2) = 0
\]
for every $q_2 \in \R$ such that 
\[
	q_2 > \max\left\{1,q^*\left(\alpha_{\infty},\beta_{\infty},s\right)\right\}
	=\max\left\{1,2\left(1+\frac{b-a}{N-2s}\right) \right\}.
\]
Now, if $a>-2s$, then we can take $q_1=q_2=q$ and get the compact embedding $H^s_{V,\rad}(\mathbb{R}^N) \hookrightarrow L^q_K(\R^N)$ for every \(q\) such that 
\[
	\max\left\{1,2\left(1+\frac{b-a}{N-2s}\right) \right\}<q<
	2\left(1+\frac{b + 2s }{N-2s}\right).
\]
Otherwise, if $a\leq -2s$, then we get the compact embedding $H^s_{V,\rad}(\mathbb{R}^N) \hookrightarrow L^{q_1}_K(\mathbb{R}^N) + L^{q_2}_K\mathbb{R}^N)$ for every $q_1$ and $q_2$ such that 
\[
	1 < q_1 <2\left(1+\frac{b + 2s }{N-2s}\right) 
	\leq 2\left(1+\frac{b-a}{N-2s}\right) < q_2.
\]
\end{Ex}

\begin{Ex}[zero $V$ potential]
Our results also hold true if $V \equiv 0$, by taking $\beta_0=\beta_{\infty}=0$. 
In such a case, we get the compact embedding $H^s_{V,\rad} (\mathbb{R}^N)\hookrightarrow L^{q_1}_K(\mathbb{R}^N)+L^{q_2}_K(\mathbb{R}^N)$ for every $q_1$ and $q_2$ such that 
\[
1 < q_1 <2\frac{N + \alpha_0}{N-2s}
\quad\textrm{and}\quad 
q_2 > \max\left\{1,2\frac{N + \alpha_{\infty}}{N-2s} \right\},
\]
provided that there exist two numbers $\alpha_0 > -\frac{N}{2}-s \,\left( =\alpha^*_s(0)\right) $ and $\alpha_{\infty}\in\R$ such that 
\[ 
\limsup_{r\to 0} \frac{K(r)}{r^{\alpha_0}} < +\infty 
\quad\textrm{and}\quad 
\limsup_{r\to\infty} \frac{K(r)}{r^{\alpha_{\infty}}} < +\infty.
\]
If the latter assumption holds true for some $\alpha_{\infty}<\alpha_0$, then we get compact embeddings into single Lebesgue spaces.

Observe that, in general, it is convenient to choose $\alpha_0$ large and $<\alpha_{\infty}$ small. So, if for example $K(r)=r^b$ with $b > -\frac{N}{2}-s$, then we will pick $\alpha_0=\alpha_{\infty}=b$ and thus we get the compact embedding $H^s_{V,\rad}(\mathbb{R}^N) \hookrightarrow L^{q_1}_K(\mathbb{R}^N)+L^{q_2}_K(\mathbb{R}^N)$ for every $q_1$ and $q_2$ such that 
\[
1 < q_1 <2\frac{N + b}{N-2s}< q_2.
\]
\end{Ex}

\begin{Ex}[exponenatial type potentials]
Consider the following potentials:
\[
	V(r)=e^{2r} \quad \text{ and } \quad K(r)=e^r.
\]
As to assumption \eqref{esssup:S0}, we have that
\[
	\sup_{r \in (0,R_1)} \frac{K(r)}{r^{\alpha_0}V(r)^{\beta_0}} = \sup_{r \in (0,R_1)} \frac{e^r}{r^{\alpha_0}e^{2r\beta_0}} = \sup_{r \in (0,R_1)} \frac{e^{r(1-2\beta_0)}}{r^{\alpha_0}}
\]
is finite if and only if $\alpha_0 \leq 0$ and $\beta_0 \in [0,1]$. Since $q^*(\alpha,\beta,s)$ is increasing in $\alpha$ and decreasing in $\beta$ (see Remark \ref{rem:alpha*:q*}) we choose $\alpha_0=\beta_0=0$ and obtain that
\[
	\lim_{R_1 \to 0^+} \cS_0(q_1,R_1) = 0
\]
for every $q_1 \in \R$ such that 
\begin{equation}
\label{Ex3:q1}
	1 <  q_1 < q^*(0,0,s)=\frac{2N}{N-2s}=2^*_s.
\end{equation}
As to condition \eqref{esssup:Sinf}, we have  
\[
	\sup_{r > R_2} \frac{K(r)}{r^{\alpha_{\infty}}V(r)^{\beta_{\infty}}} = \sup_{r > R_2} \frac{e^r}{r^{\alpha_{\infty}}e^{2r\beta_{\infty}}} = \sup_{r > R_2} \frac{e^{r(1-2\beta_{\infty})}}{r^{\alpha_{\infty}}}
\]
which is finite if and only if $\beta_{\infty} = \frac12$ and $\alpha_{\infty} \geq 0$, or $\frac12<\beta_{\infty} \leq 1$ and $\alpha_{\infty} \in\R $. 
In the first case, it is convenient to choose $\alpha_{\infty} = 0$ (cf. Remark \ref{rem:alpha*:q*}) and we get that 
\begin{equation}
\label{Ex3:q2}
\lim_{R_2 \to +\infty} \cS_{\infty}(q_2,R_2) = 0
\end{equation}
for every $q_2 \in \R$ such that 
\begin{equation}
	\label{Ex3:q2_1}
	q_2 > \max\left\{1,q^*\left(0,\frac12,s\right)\right\}
	=2\frac{N-s}{N-2s}.
\end{equation}
In the second case, it is convenient to choose $\beta_{\infty}=1$ and $\alpha_{\infty}$ as small as possible (see again Remark \ref{rem:alpha*:q*}). Then, taking any $\alpha_{\infty}\leq 0$, we get \eqref{Ex3:q2} for every $q_2 \in \R$ such that 
\begin{equation}
	\label{Ex3:q2_2}
	q_2 > \max\left\{1,2,q^*\left(\alpha_{\infty},1,s\right)\right\}
	\max\left\{1,2,2\frac{\alpha_{\infty}+N-2s}{N-2s}\right\}
	=2,
\end{equation}
which is a better result than \eqref{Ex3:q2_1}, where the right-hand side is larger than $2$. As a conclusion, comparing \eqref{Ex3:q1} and \eqref{Ex3:q2_2}, we have a compact embedding $H^s_{V,\rad}(\mathbb{R}^N) \hookrightarrow L^{q}_K(\mathbb{R}^N)$ for every $2<q<2^*_s$.
%
\end{Ex}

\begin{Ex}[mixed type potentials]
Consider the following potentials:
\[
	V(r)=e^{-ar} \quad \text{ and } \quad  K(r)=r^de^{-br}
\]
with $a,b>0$ and $d \in \R$. Then
\[
	\sup_{r \in (0,R_1)} \frac{K(r)}{r^{\alpha_0}V(r)^{\beta_0}} = \sup_{r \in (0,R_1)} \frac{r^de^{-br}}{r^{\alpha_0}e^{-ar\beta_0}} = \sup_{r \in (0,R_1)} \frac{r^{d-\alpha_0}e^{-br}}{e^{-ar\beta_0}}
\]
is finite if and only if $\alpha_0 \geq d$ and $\beta_0 \in [0,1]$, so that we can choose $\beta_0=0$ and $\alpha_0$ as large as we like, in order that Theorem \ref{S0:holdness} yields  
\[
	\lim_{R_1 \to 0^+} \cS_0(q_1,R_1) = 0
\]
for every $q_1 >1 $. 
As to assumption \eqref{esssup:Sinf}, letting $\beta_{\infty}=0$ and $\alpha_{\infty}=-N$ we get 
\[
\sup_{r > R_2} \frac{K(r)}{r^{\alpha_{\infty}}V(r)^{\beta_{\infty}}} 
= \sup_{r > R_2} r^{d+N} e^{-br}
= 0
\]
and therefore Theorem \ref{Sinf:holdness} gives 
\[
\lim_{R_2 \to +\infty} \cS_{\infty}(q_2,R_2) = 0
\]
for every 
\[
	\label{Ex5:q2}
	q_2 > \max\left\{1,0,q^*(-N,0,s)\right\}
	=\max\left\{1,0,0\right\}=1.
\]
As a conclusion, Theorem \ref{th:comp:emb} ensures that the embedding $H^s_{V,\rad}(\mathbb{R}^N) \hookrightarrow L^{q}_K (\mathbb{R}^N)$ is compact for every $q>1$.
%
%
%
\end{Ex}


\bibliographystyle{abbrv}
\bibliography{Bibliography}

\begin{thebibliography}{10}

\bibitem{AlAlMe2014}
F.~S.~B. Albuquerque, C.~O. Alves, and E.~S. Medeiros.
\newblock Nonlinear {S}chr\"{o}dinger equation with unbounded or decaying
  radial potentials involving exponential critical growth in {$\R^2$}.
\newblock {\em J. Math. Anal. Appl.}, 409(2):1021--1031, 2014.

\bibitem{AlSo2013}
C.~O. Alves and M.~A.~S. Souto.
\newblock Existence of solutions for a class of nonlinear {S}chr\"{o}dinger
  equations with potential vanishing at infinity.
\newblock {\em J. Differential Equations}, 254(4):1977--1991, 2013.

\bibitem{AmFeMa2005}
A.~Ambrosetti, V.~Felli, and A.~Malchiodi.
\newblock Ground states of nonlinear {S}chr\"{o}dinger equations with
  potentials vanishing at infinity.
\newblock {\em J. Eur. Math. Soc. (JEMS)}, 7(1):117--144, 2005.

\bibitem{AmMaSe2001}
A.~Ambrosetti, A.~Malchiodi, and S.~Secchi.
\newblock Multiplicity results for some nonlinear {S}chr\"{o}dinger equations
  with potentials.
\newblock {\em Arch. Ration. Mech. Anal.}, 159(3):253--271, 2001.

\bibitem{AmRa1973}
A.~Ambrosetti and P.~H. Rabinowitz.
\newblock Dual variational methods in critical point theory and applications.
\newblock {\em J. Functional Analysis}, 14:349--381, 1973.

\bibitem{AmRu2006}
A.~Ambrosetti and D.~Ruiz.
\newblock Radial solutions concentrating on spheres of nonlinear
  {S}chr\"{o}dinger equations with vanishing potentials.
\newblock {\em Proc. Roy. Soc. Edinburgh Sect. A}, 136(5):889--907, 2006.

\bibitem{AmWa2005}
A.~Ambrosetti and Z.-Q. Wang.
\newblock Nonlinear {S}chr\"{o}dinger equations with vanishing and decaying
  potentials.
\newblock {\em Differential Integral Equations}, 18(12):1321--1332, 2005.

\bibitem{Ambrosio2016}
V.~Ambrosio.
\newblock Ground states for superlinear fractional {S}chr\"{o}dinger equations
  in {$\R^N$}.
\newblock {\em Ann. Acad. Sci. Fenn. Math.}, 41(2):745--756, 2016.

\bibitem{Ambrosio2020}
V.~Ambrosio.
\newblock {\em Nonlinear fractional {S}chr\"{o}dinger equations in {$\R^N$}}.
\newblock Frontiers in Elliptic and Parabolic Problems.
  Birkh\"{a}user/Springer, Cham, [2021] \copyright 2021.

\bibitem{ApMBSe}
L.~Appolloni, G.~M. Bisci, and S.~Secchi.
\newblock A note on the nls equation on cartan-hadamard manifolds with
  unbounded and vanishing potentials, 2023.

\bibitem{Azzollini2008}
A.~Azzollini.
\newblock A multiplicity result for a semilinear {M}axwell type equation.
\newblock {\em Topol. Methods Nonlinear Anal.}, 31(1):83--110, 2008.

\bibitem{AzPiPo2011}
A.~Azzollini, L.~Pisani, and A.~Pomponio.
\newblock Improved estimates and a limit case for the electrostatic
  {K}lein-{G}ordon-{M}axwell system.
\newblock {\em Proc. Roy. Soc. Edinburgh Sect. A}, 141(3):449--463, 2011.

\bibitem{AzPo2008}
A.~Azzollini and A.~Pomponio.
\newblock Compactness results and applications to some ``zero mass'' elliptic
  problems.
\newblock {\em Nonlinear Anal.}, 69(10):3559--3576, 2008.

\bibitem{BaBeRo2007}
M.~Badiale, V.~Benci, and S.~Rolando.
\newblock A nonlinear elliptic equation with singular potential and
  applications to nonlinear field equations.
\newblock {\em J. Eur. Math. Soc. (JEMS)}, 9(3):355--381, 2007.

\bibitem{BaBeRo2009}
M.~Badiale, V.~Benci, and S.~Rolando.
\newblock Three dimensional vortices in the nonlinear wave equation.
\newblock {\em Boll. Unione Mat. Ital. (9)}, 2(1):105--134, 2009.

\bibitem{BaGuRo2015-1}
M.~Badiale, M.~Guida, and S.~Rolando.
\newblock Compactness and existence results in weighted {S}obolev spaces of
  radial functions. {P}art {I}: compactness.
\newblock {\em Calc. Var. Partial Differential Equations}, 54(1):1061--1090,
  2015.

\bibitem{BaGuRo2015-2}
M.~Badiale, M.~Guida, and S.~Rolando.
\newblock Compactness and existence results in weighted {S}obolev spaces of
  radial functions. {P}art {II}: existence.
\newblock {\em NoDEA Nonlinear Differential Equations Appl.}, 23(6):Art. 67,
  34, 2016.

\bibitem{BaGuRo2021}
M.~Badiale, M.~Guida, and S.~Rolando.
\newblock Compactness and existence results for quasilinear elliptic problems
  with singular or vanishing potentials.
\newblock {\em Anal. Appl. (Singap.)}, 19(5):751--777, 2021.

\bibitem{BaPiRo2011}
M.~Badiale, L.~Pisani, and S.~Rolando.
\newblock Sum of weighted {L}ebesgue spaces and nonlinear elliptic equations.
\newblock {\em NoDEA Nonlinear Differential Equations Appl.}, 18(4):369--405,
  2011.

\bibitem{BaWa1995}
T.~Bartsch and Z.~Q. Wang.
\newblock Existence and multiplicity results for some superlinear elliptic
  problems on {$\R^N$}.
\newblock {\em Comm. Partial Differential Equations}, 20(9-10):1725--1741,
  1995.

\bibitem{BeFo2004-1}
V.~Benci and D.~Fortunato.
\newblock A strongly degenerate elliptic equation arising from the semilinear
  {M}axwell equations.
\newblock {\em C. R. Math. Acad. Sci. Paris}, 339(12):839--842, 2004.

\bibitem{BeFo2004-2}
V.~Benci and D.~Fortunato.
\newblock Towards a unified field theory for classical electrodynamics.
\newblock {\em Arch. Ration. Mech. Anal.}, 173(3):379--414, 2004.

\bibitem{BeLi1983-1}
H.~Berestycki and P.-L. Lions.
\newblock Nonlinear scalar field equations. {I}. {E}xistence of a ground state.
\newblock {\em Arch. Rational Mech. Anal.}, 82(4):313--345, 1983.

\bibitem{BeLi1983-2}
H.~Berestycki and P.-L. Lions.
\newblock Nonlinear scalar field equations. {II}. {E}xistence of infinitely
  many solutions.
\newblock {\em Arch. Rational Mech. Anal.}, 82(4):347--375, 1983.

\bibitem{BoMe2011}
D.~Bonheure and C.~Mercuri.
\newblock Embedding theorems and existence results for nonlinear
  {S}chr\"{o}dinger-{P}oisson systems with unbounded and vanishing potentials.
\newblock {\em J. Differential Equations}, 251(4-5):1056--1085, 2011.

\bibitem{BrNi}
H.~Br\'{e}zis and L.~Nirenberg.
\newblock Positive solutions of nonlinear elliptic equations involving critical
  {S}obolev exponents.
\newblock {\em Comm. Pure Appl. Math.}, 36(4):437--477, 1983.

\bibitem{CaCePeUb2021}
J.~A. Cardoso, P.~Cerda, D.~Pereira, and P.~Ubilla.
\newblock Schr\"{o}dinger equations with vanishing potentials involving
  {B}rezis-{K}amin type problems.
\newblock {\em Discrete Contin. Dyn. Syst.}, 41(6):2947--2969, 2021.

\bibitem{Cerami1978}
G.~Cerami.
\newblock Un criterio di esistenza per i punti critici su varieta'illimitate.
\newblock 1978.

\bibitem{CoTa2004}
A.~Cotsiolis and N.~K. Tavoularis.
\newblock Best constants for {S}obolev inequalities for higher order fractional
  derivatives.
\newblock {\em J. Math. Anal. Appl.}, 295(1):225--236, 2004.

\bibitem{DeNapoli2018}
P.~L. De~N\'{a}poli.
\newblock Symmetry breaking for an elliptic equation involving the fractional
  {L}aplacian.
\newblock {\em Differential Integral Equations}, 31(1-2):75--94, 2018.

\bibitem{dPFe1996}
M.~del Pino and P.~L. Felmer.
\newblock Local mountain passes for semilinear elliptic problems in unbounded
  domains.
\newblock {\em Calc. Var. Partial Differential Equations}, 4(2):121--137, 1996.

\bibitem{DiPaVa2012}
E.~Di~Nezza, G.~Palatucci, and E.~Valdinoci.
\newblock Hitchhiker's guide to the fractional {S}obolev spaces.
\newblock {\em Bull. Sci. Math.}, 136(5):521--573, 2012.

\bibitem{DPPaVa2013}
S.~Dipierro, G.~Palatucci, and E.~Valdinoci.
\newblock Existence and symmetry results for a {S}chr\"{o}dinger type problem
  involving the fractional {L}aplacian.
\newblock {\em Matematiche (Catania)}, 68(1):201--216, 2013.

\bibitem{FeQuTa2012}
P.~Felmer, A.~Quaas, and J.~Tan.
\newblock Positive solutions of the nonlinear {S}chr\"{o}dinger equation with
  the fractional {L}aplacian.
\newblock {\em Proc. Roy. Soc. Edinburgh Sect. A}, 142(6):1237--1262, 2012.

\bibitem{MoRaSe2016}
G.~Molica~Bisci, V.~D. Radulescu, and R.~Servadei.
\newblock {\em Variational methods for nonlocal fractional problems}, volume
  162 of {\em Encyclopedia of Mathematics and its Applications}.
\newblock Cambridge University Press, Cambridge, 2016.
\newblock With a foreword by Jean Mawhin.

\bibitem{Nehari1960}
Z.~Nehari.
\newblock On a class of nonlinear second-order differential equations.
\newblock {\em Trans. Amer. Math. Soc.}, 95:101--123, 1960.

\bibitem{Palais1979}
R.~S. Palais.
\newblock The principle of symmetric criticality.
\newblock {\em Comm. Math. Phys.}, 69(1):19--30, 1979.

\bibitem{PaSm1964}
R.~S. Palais and S.~Smale.
\newblock A generalized {M}orse theory.
\newblock {\em Bull. Amer. Math. Soc.}, 70:165--172, 1964.

\bibitem{PaPi2014}
G.~Palatucci and A.~Pisante.
\newblock Improved {S}obolev embeddings, profile decomposition, and
  concentration-compactness for fractional {S}obolev spaces.
\newblock {\em Calc. Var. Partial Differential Equations}, 50(3-4):799--829,
  2014.

\bibitem{Rabinowitz1974}
P.~H. Rabinowitz.
\newblock Variational methods for nonlinear eigenvalue problems.
\newblock In {\em Eigenvalues of non-linear problems ({C}entro {I}nternaz.
  {M}at. {E}stivo ({C}.{I}.{M}.{E}.), {III} {C}iclo, {V}arenna, 1974)}, pages
  139--195. 1974.

\bibitem{Rabinowitz1978}
P.~H. Rabinowitz.
\newblock Some critical point theorems and applications to semilinear elliptic
  partial differential equations.
\newblock {\em Ann. Scuola Norm. Sup. Pisa Cl. Sci. (4)}, 5(1):215--223, 1978.

\bibitem{Rabinowitz1992}
P.~H. Rabinowitz.
\newblock On a class of nonlinear {S}chr\"{o}dinger equations.
\newblock {\em Z. Angew. Math. Phys.}, 43(2):270--291, 1992.

\bibitem{Secchi2013}
S.~Secchi.
\newblock Ground state solutions for nonlinear fractional {S}chr\"{o}dinger
  equations in {$\R^N$}.
\newblock {\em J. Math. Phys.}, 54(3):031501, 17, 2013.

\bibitem{Secchi2016}
S.~Secchi.
\newblock On fractional {S}chr\"{o}dinger equations in {$\R^N$} without the
  {A}mbrosetti-{R}abinowitz condition.
\newblock {\em Topol. Methods Nonlinear Anal.}, 47(1):19--41, 2016.

\bibitem{Snitzoff2003}
P.~Sintzoff.
\newblock Symmetry of solutions of a semilinear elliptic equation with
  unbounded coefficients.
\newblock {\em Differential Integral Equations}, 16(7):769--786, 2003.

\bibitem{Strauss1977}
W.~A. Strauss.
\newblock Existence of solitary waves in higher dimensions.
\newblock {\em Comm. Math. Phys.}, 55(2):149--162, 1977.

\bibitem{SuTi2012}
J.~Su and R.~Tian.
\newblock Weighted {S}obolev type embeddings and coercive quasilinear elliptic
  equations on {$\R^N$}.
\newblock {\em Proc. Amer. Math. Soc.}, 140(3):891--903, 2012.

\bibitem{SuWaWi2007-1}
J.~Su, Z.-Q. Wang, and M.~Willem.
\newblock Nonlinear {S}chr\"{o}dinger equations with unbounded and decaying
  radial potentials.
\newblock {\em Commun. Contemp. Math.}, 9(4):571--583, 2007.

\bibitem{SuWaWi2007-2}
J.~Su, Z.-Q. Wang, and M.~Willem.
\newblock Weighted {S}obolev embedding with unbounded and decaying radial
  potentials.
\newblock {\em J. Differential Equations}, 238(1):201--219, 2007.

\end{thebibliography}
%
%
%
%
%
%
%
\end{document}